\documentclass[]{aspm}
\articleinfo{}{}{}

\usepackage{amssymb}
\usepackage{amsbsy}
\usepackage{amscd}
\usepackage{amsmath}
\usepackage{amsthm}
\usepackage{amstext}
\usepackage{mathrsfs}  % allows nice script font for sheaves
\usepackage{bm}        % allows bold italic letters in math mode
\usepackage{mathtools} % allows more extendible arrows

\usepackage[all]{xy}
\SelectTips{cm}{10}

\theoremstyle{plain}
\newtheorem{theorem}{Theorem}
\newtheorem*{theorem*}{Theorem}
\newtheorem{prop}{Proposition}\numberwithin{prop}{section}
\newtheorem{lemma}{Lemma}\numberwithin{lemma}{section}
\numberwithin{cor}{section}
\newtheorem{question}{Question}\numberwithin{question}{section}
\newtheorem*{hypothesis*}{Hypothesis}

\theoremstyle{definition}
\newtheorem{defin}{Definition}\numberwithin{defin}{section}
\newtheorem{example}{Example}\numberwithin{example}{section}

\theoremstyle{remark}
\newtheorem{remark}{Remark}\numberwithin{remark}{section}

\newcommand{\Z}{\mathbb Z}
\newcommand{\Q}{\mathbb Q}
\newcommand{\R}{\mathbb R}
\newcommand{\A}{\mathbb A}
\renewcommand{\P}{\mathbb P}

\newcommand{\Gm}{\mathbb{G}_{\mathrm{m}}}

\newcommand{\mult}{^{\times}}
\newcommand{\dual}{^{\vee}}
\newcommand{\tensor}{\otimes}
\newcommand{\linedef}[1]{\emph{#1}}
\newcommand{\muu}{\bm{\mu}}
\DeclareMathOperator{\Br}{\mathrm{Br}}
\newcommand{\isom}{\cong}
\newcommand{\et}{\mathrm{\acute{e}t}}
\newcommand{\Het}{H_{\et}}
\DeclareMathOperator{\Pic}{\mathrm{Pic}}
\DeclareMathOperator{\Spec}{\mathrm{Spec}}
\newcommand{\id}{\mathrm{id}}
\DeclareMathOperator{\im}{\mathrm{im}}

\newcommand{\sheaf}[1]{\mathscr{#1}}
\newcommand{\LL}{\sheaf{L}}
\newcommand{\OO}{\sheaf{O}}

\newcommand{\EE}{\sheaf{E}}

\newcommand{\NN}{\sheaf{N}}
\newcommand{\VV}{\sheaf{V}}

\newcommand{\CliffAlg}{\mathscr{C}}

\DeclareMathOperator{\Tr}{\mathrm{Tr}}
\DeclareMathOperator{\Hom}{\mathrm{Hom}}
\newcommand{\mapto}[1]{\xrightarrow{#1}}
\newcommand{\vp}{\varphi}

\newcommand{\tot}{\bullet}
\newcommand{\ur}{\mathrm{ur}}
\newcommand{\unram}[1]{\mathcal{#1}}
\newcommand{\KMil}{K_{\text{M}}}
\newcommand{\KMilur}{K_{\text{M},\ur}}
\newcommand{\KMiltot}{\KMil^{\tot}}
\newcommand{\KMilurtot}{\KMilur^{\tot}}
\newcommand{\KKMil}{\unram{K}_{\text{M}}}
\newcommand{\KKMiltot}{\KKMil^{\tot}}

\newcommand{\Hur}{H_{\ur}}
\newcommand{\Iur}{I_{\ur}}
\newcommand{\qform}[1]{<\! #1 \!>}
\newcommand{\Pform}[1]{\ll\! #1 \!\gg}
\newcommand{\Brtwo}{{}_2\!\Br}
\newcommand{\Zar}{\text{Zar}}

\newcommand{\res}{\partial}

\newcommand{\Cat}[1]{\mathsf{#1}}
\newcommand{\Field}{\Cat{Field}}
\newcommand{\Ring}{\Cat{Ring}}
\newcommand{\Sch}{\Cat{Sch}}
\newcommand{\Sm}{\Cat{Sm}}
\newcommand{\Ab}{\Cat{Ab}}
\newcommand{\Abtot}{\Cat{Ab}^{\tot}}
\newcommand{\Local}{\Cat{Local}}
\newcommand{\cd}{\mathit{cd}}
\newcommand{\vcd}{\mathit{vcd}}
\newcommand{\total}{\mathrm{tot}}

\title[Remarks on the Milnor conjecture over schemes]
{Remarks on the Milnor conjecture over schemes}

\author{Asher Auel}

\address{%
Department of Mathematics \& CS \\ %
Emory University \\ %
400 Dowman Drive NE W401 \\ %
Atlanta, GA 30322}

\email{auel@mathcs.emory.edu}



\usepackage{hyperref}
\hypersetup{pdftitle={Remarks on the Milnor conjecture for schemes}}
\hypersetup{pdfauthor={Asher Auel}}
\hypersetup{colorlinks=true,linkcolor=blue,anchorcolor=blue,citecolor=blue,letterpaper=true}

\begin{document}

\begin{abstract}
% Merkurjev's theorem---that every 2-torsion Brauer class is represented
% by the Clifford algebra of a quadratic form---is in general false when
% the base is no longer a field.  Parimala, Scharlau, and Sridharan
% found smooth complete $p$-adic curves for which Merkurjev's theorem is
% equivalent to the existence of a rational theta characteristic.  We
% prove that on a smooth proper $\Qp$-curve $X$, replacing the Witt
% group by the total Grothendieck--Witt group of line bundle-valued
% quadratic forms, these obstructions vanish; every 2-torsion Brauer
% class on $X$ is represented by the Clifford invariant of some line
% bundle-valued quadratic form.  The key new ingrediant is the actual
% construction of the Clifford invariant of a line bundle-valued
% quadratic form, when the value line bundle is nontrivial.
The Milnor conjecture has been a driving force in the theory of
quadratic forms over fields, guiding the development of the theory of
cohomological invariants, ushering in the theory of motivic
cohomology, and touching on questions ranging from sums of squares to
the structure of absolute Galois groups.  Here, we survey some recent
work on generalizations of the Milnor conjecture to the context of
schemes (mostly smooth varieties over fields of characteristic $\neq
2$).  Surprisingly, a version of the Milnor conjecture fails to hold
for certain smooth complete $p$-adic curves with no rational theta
characteristic (this is the work of Parimala, Scharlau, and
Sridharan).  We explain how these examples fit into the larger context
of the unramified Milnor question, offer a new approach to the
question, and discuss new results in the case of curves over local
fields and surfaces over finite fields.
\end{abstract}

\maketitle

The first cases of the (as of yet unnamed) Milnor conjecture were
studied in Pfister's Habilitationsschrift
\cite{pfister:habilitationsschrift} in the mid 1960s.  As Pfister
\cite[p.\ 3]{pfister:historical_remarks} himself points out, ``[the
Milnor conjecture] stimulated research for quite some time.''  Indeed,
it can be seen as one of the driving forces in the theory of quadratic
forms since Milnor's original formulation \cite{milnor:conjecture} in
the early 1970s.

The classical cohomological invariants of quadratic forms (rank,
discriminant, and Clifford--Hasse--Witt invariant) have a deep
connection with the history and development of the subject.  In
particular, they are used in the classification (Hasse--Minkowski
local-global theorem) of quadratic forms over local and global fields.
The first ``higher invariant'' was described in Arason's thesis
\cite{arason:quadratic_forms_invariant},
\cite{arason:cohomological_invariants}.
%, and is also used in classification theorems.  
The celebrated results of Merkurjev \cite{merkurjev:degree_2} and
Merkurjev--Suslin \cite{merkurjev_suslin:norm3} settled special cases
of the Milnor conjecture in the early 1980s, and served as a starting
point for Voevodsky's development of the theory of motivic cohomology.
Other special cases were settled by Arason--Elman--Jacob
\cite{arason_elman_jacob:graded_Witt_I} and Jacob--Rost
\cite{jacob_rost:invariant_degree_4}.  Voevodsky's motivic cohomology
techniques \cite{voevodsky:Milnor_conjecture_I} ultimately led to a
complete solution of the Milnor conjecture, for which he was awarded
the 2002 Fields Medal.
% During this period of development, the
% use of intersection theory, $K$-cohomology, and Chow motivic
% decomposition became increasingly important, finally culminating in
% the recent book of Elman--Karpenko--Merkurjev \cite{EKM}.

% Classical invariants, Arason
% invariant, search for higher invariants.  
The consideration of quadratic forms over rings (more general than
fields) has its roots in the number theoretic study of lattices (i.e.\
quadratic forms over $\Z$) by Gauss as well as the algebraic study of
division algebras and hermitian forms (i.e.\ quadratic forms over
algebras with involution) by Albert.  A general framework for the
study of quadratic forms over rings was established by Bass
\cite{bass:algebraic_K-theory}, with the case of (semi)local rings
treated in depth by Baeza \cite{baeza:semilocal_rings}.  Bilinear
forms over Dedekind domains (i.e.\ unimodular lattices) were studied
in a number theoretic context by Fr\"ohlich
\cite{frohlich:unimodular}, while the consideration of quadratic forms
over algebraic curves (and their function fields) was initiated by
Geyer, Harder, Knebusch, Scharlau \cite{harder:Witt_group_curves},
\cite{geyer_harder_knebusch_scharlau},
\cite{knebusch_scharlau:reciprocity},
\cite{knebusch:curves_real_fields_II}.  The theory of quadratic (and
bilinear) forms over schemes was developed by Knebusch
\cite{knebusch:heidelberg}, \cite{knebusch}, and utilized by Arason
\cite{arason:der_wittring}, Dietel
\cite{dietel:Witt_ring_real_curves}, Parimala
\cite{parimala:Witt_groups_conics},
\cite{parimala:curves_local_fields}, Fern\'andez-Carmena
\cite{fernandez-carmena:Witt_group_surfaces}, Sujatha
\cite{sujatha:Witt_groups_real_surfaces},
\cite{parimala_sujatha:witt_group_hyperelliptic_curve},
Arason--Elman--Jacob
\cite{arason_elman_jacob:Witt_ring_elliptic_curve},
\cite{arason_elman_jacob:Witt_ring_elliptic_curve_local_field}, and
others.  A theory of symmetric bilinear forms in additive and abelian
categories was developed by Quebbemann--Scharlau--Schulte
\cite{quebbemann_scharlau_et_al}, \cite{quebbemann_scharlau_schulte}.
Further enrichment came eventually from the triangulated category
techniques of Balmer \cite{balmer:derived_witt_groups},
\cite{balmer:triangular_witt_groups_I},
\cite{balmer:triangular_witt_groups_II}, and Walter
\cite{walter:GW_groups_triangulated_categories}.  This article will
focus on progress in generalizing the Milnor conjecture to these
contexts.

These remarks grew out of a lecture at the RIMS-Camp-Style seminar
``Galois-theoretic Arithmetic Geometry'' held October 19-23, 2010, in
Kyoto, Japan.  The author would like to thank the organizers for their
wonderful hospitality during that time.  He would also like to thank
Stefan Gille, Moritz Kerz, R.\ Parimala, and V.\ Suresh for many
useful conversations.  The author acknowledges the generous support of
the Max Plank Institute for Mathematics in Bonn, Germany where this
article was written under excellent working conditions.  This author
is also partially supported by National Science Foundation grant MSPRF
DMS-0903039.

\subsubsection*{Conventions.} A graded abelian group or ring
$\prod_{n\geq 0} M^n$ will be denoted by $M^{\tot}$.  If $0 \subset
\dotsm \subset N^2 \subset N^1 \subset N^0 = M$ is a decreasing
filtration of a ring $M$ by ideals, denote by $N^{\tot}/N^{\tot+1} =
\prod_{n\geq 0}N^{n}/N^{n+1}$ the associated graded ring.  Denote by
${}_2 M$ the elements of order 2 in an abelian group $M$. All abelian
groups will be written additively.

\section{The Milnor conjecture over fields}
\label{sec:fields}

Let $F$ be a field of characteristic $\neq 2$.  The total Milnor
$K$-ring $\KMiltot(F) = T^{\tot}(F\mult)/\langle a \tensor (1-a)\, :
\, a\ \in F\mult\rangle$ was introduced in \cite{milnor:conjecture}.
The total Galois cohomology ring $H^{\tot}(F,\muu_2^{\tensor \tot}) =
\bigoplus_{n\geq 2} H^n(F,\muu_2^{\tensor n})$ is canonically
isomorphic, under our hypothesis on the characteristic of $F$, to the
total Galois cohomology ring $H^{\tot}(F,\Z/2\Z)$ with coefficients in
the trivial Galois module $\Z/2\Z$.  The Witt ring $W(F)$ of
nondegenerate quadratic forms modulo hyperbolic forms has an
decreasing filtration $0\subset\dotsm\subset I^1(F) \subset I^0(F)=
W(F)$ generated by powers of the \linedef{fundamental ideal} $I(F)$ of
even rank forms.  The Milnor conjecture relates these three objects:
Milnor $K$-theory, Galois cohomology, and quadratic forms.

The quotient map $\KMil^1(F)=F\mult \to F\mult/F\mult{}^2 \isom
H^1(X,\muu_2)$ induces a graded ring homomorphism $h^{\tot} :
\KMiltot(F)/2 \to H^{\tot}(F,\muu_2^{\tensor \tot})$ called the
\linedef{norm residue symbol} by Bass--Tate
\cite{bass_tate:Milnor_global_field}.  The \linedef{Pfister form} map
$\KMil^1(F) = F\mult \to I(F)$ given by $a \mapsto \Pform{a} =
\qform{1,-a}$ induces a group homomorphism $\KMil^1(F)/2 \to
I^1(F)/I^2(F)$ (see Scharlau \cite[2 Lemma 12.10]{scharlau:book}), which
extends to a surjective graded ring homomorphism $s^{\tot} :
\KMiltot(F)/2 \to I^{\tot}(F)/I^{\tot+1}(F)$, see Milnor
\cite[Thm.\ 4.1]{milnor:conjecture}.

\begin{theorem}[Milnor conjecture]
Let $F$ be a field of characteristic $\neq 2$.  There exists a graded
ring homomorphism $e^{\tot} : I^{\tot}(F)/I^{\tot+1}(F) \to
H^{\tot}(F,\muu_2^{\tensor \tot})$ called the \linedef{higher
invariants} of quadratic forms, which fits into the following diagram
$$
\xymatrix{
\KMiltot(F)/2 \ar[r]^{h^{\tot}} \ar[d]_{s^{\tot}}& H^{\tot}(F,\muu_2^{\tensor\tot}) \\
I^{\tot}(F)/I^{\tot+1}(F) \ar[ur]_{e^{\tot}} & }
$$
of isomorphisms of graded rings.
\end{theorem}

Many excellent introductions to the Milnor conjecture and its proof
exist in the literature.  For example, see the surveys of Kahn
\cite{kahn:milnor_conjecture_article}, Friedlander--Rapoport--Suslin
\cite{friedlander_rapoport_suslin:fields_medal_2002}, Friedlander
\cite{friedlander:motivic_complexes}, Pfister \cite{pfister:survey},
and Morel \cite{morel:voevodsky}.
% This present text will focus on the generalizations of the Milnor
% conjecture from fields to schemes.

The conjecture breaks up naturally into three parts: the conjecture
for the norm residue symbol $h^{\tot}$, the conjecture for the Pfister
form map $s^{\tot}$, and the conjecture for the higher invariants
$e^{\tot}$.  Milnor \cite[Question 4.3, \S6]{milnor:conjecture}
originally made the conjecture for $h^{\tot}$ and $s^{\tot}$, which
was already known for finite, local, global, and real closed fields,
see \cite[Lemma 6.2]{milnor:conjecture}.  For general fields, the
conjecture for $h^1$ follows from Hilbert's theorem 90, and for $s^1$
and $e^1$ by elementary arguments.  The conjecture for $s^{2}$ is
easy, see Pfister \cite{pfister:habilitationsschrift}.  Merkurjev
\cite{merkurjev:degree_2} proved the conjecture for $h^{2}$ (hence for
$e^2$ as well), with alternate proofs given by Arason
\cite{arason:proof_Merkurjev_theorem}, Merkurjev
\cite{merkurjev:another_proof_norm_2}, and Wadsworth
\cite{wadsworth:proof_Merkurjev_theorem}.  The conjecture for $h^{3}$
was settled by Merkurjev--Suslin \cite{merkurjev_suslin:norm3} (and
independently by Rost \cite{rost:K_3}).  The conjecture for $e^{\tot}$
can be divided into two parts: to show the existence of maps $e^n :
I^n(F) \to H^n(X,\muu_2^{\tensor n})$ (which are \emph{a priori} only
defined on generators, the Pfister forms), and then to show they are
surjective.  The existence of $e^3$ was proved by Arason
\cite{arason:quadratic_forms_invariant},
\cite{arason:cohomological_invariants}.  The existence of $e^4$ was
proved by Jacob--Rost \cite{jacob_rost:invariant_degree_4} and
independently Szyjewski \cite{shyevski:fifth_invariant}.  Voevodsky
\cite{voevodsky:Milnor_conjecture_I} proved the conjecture for
$h^{\tot}$.  Orlov--Vishik--Voevodsky \cite{orlov_vishik_voevodsky}
proved the conjecture for $s^{\tot}$, with different proofs given by
Morel \cite{morel:proof_milnor} and Kahn--Sujatha
\cite{kahn_sujatha:motivic_cohomology_unramified_quadrics}.

% \begin{theorem}[Milnor's Conjecture for quadratic forms]
% Let $k$ be a field of characteristic $\neq 2$.  Let $W(F)$ be the Witt
% ring of symmetric bilinear forms over $F$ and $I^*(F)$ the filtration
% induced by the powers of the \linedef{fundamental} ideal of forms of
% even rank.  For every $n \geq 0$, there exists an ``invariant'' $e^n :
% I^n(k) \to H^n(k,\Z/2\Z)$ inducing an isomorphism
% $$
% \ol{e}_n : I^n(k)/I^{n+1}(k) \isomto H^n(k,\Z/2\Z)
% $$
% of abelian groups.
% \end{theorem}

% While this conjecture was not stated explicitly by Milnor, it can be
% broken into pieces, each of which can  Milnor \cite{milnor:conjecture} statement of Milnor conjectures,
% Go through history.

% Let $\KMil^*(F)$ be the total \linedef{Milnor $K$-theory}
% ring.  

\subsection{Classical invariants of quadratic forms}
\label{subsec:Classical_invariants_of_quadratic_forms}

The theory of quadratic forms over a general field has its genesis in
Witt's famous paper \cite{witt:quadratic}.  Because of the assumption
of characteristic $\neq 2$, we do not distinguish between quadratic
and symmetric bilinear forms.  The \linedef{orthogonal sum}
$(V,b)\perp(V',b')=(V\oplus V', b+b')$ and \linedef{tensor product}
$(V,b)\tensor(V',b') = (V\tensor V',b\tensor b')$ give a semiring
structure on the set of isometry classes of symmetric bilinear forms
over $F$.  The \linedef{hyperbolic plane} is the symmetric bilinear
form $(H,h)$, where $H=F^2$ and $h((x,y),(x',y')) = xy'+x'y$.  The
\linedef{Witt ring} of symmetric bilinear forms is the quotient of the
Grothendieck ring of nondegenerate symmetric bilinear forms over $F$
with respect to $\perp$ and $\tensor$, modulo the ideal generated by
the hyperbolic plane, see Scharlau \cite[Ch.\ 2]{scharlau:book}.

The \linedef{rank} of a bilinear form $(V,b)$ is the $F$-vector space
dimension of $V$.  Since the hyperbolic plane has rank 2, the rank
modulo 2 is a well defined invariant of an element of the Witt ring,
and gives rise to a surjective ring homomorphism
$$
e^0 : W(F) = I^0(F) \to \Z/2\Z = H^0(F,\Z/2\Z)
$$
whose kernel is the \linedef{fundamental ideal} $I(F)$.

The \linedef{signed discriminant} of a non-degenerate bilinear form
$(V,b)$ is defined as follows.  Choosing an $F$-vector space basis
$v_1,\dotsc,v_r$ of $V$, we consider the Gram matrix $M_b$ of $b$,
i.e.\ the matrix given by $M_b = (b(v_i,v_j))$. %_{0 \leq i,j \leq r}$.
Then $b$ is given by the formula $b(v,w) = v^t M_b w$, where $v,w \in
F^r \isom V$.  The Gram matrix of $b$, with respect to a different
basis for $V$ with change of basis matrix $T$, is $T^tM_bT$.  Thus
$\det M_b \in F\mult$, which depends on the choice of basis, is only
well-defined up to squares.  For $a \in F\mult$, denote by $(a)$ its
class in the abelian group $F\mult/F\mult{}^2$.  The signed
discriminant of $(V,b)$ is defined as $(-1)^{r(r-1)/2} \det M_b \in
F\mult/F\mult{}^2$.  Introducing the sign into the signed discriminant
ensures its vanishing on the ideal of hyperbolic forms, hence it
descents to the Witt group.  While the signed discriminant is not
additive on $W(F)$, its restriction to $I(F)$ gives rise to a group
homomorphism
$$
e^1 : I(F) \to F\mult/F\mult{}^2 \isom H^1(F,\muu_2)
$$
which is easily seen to be surjective.  It's then not difficult to
check that its kernel coincides with the square $I^2(F)$ of the
fundamental ideal.  See Scharlau \cite[\S2.2]{scharlau:book} for more
details.

The \linedef{Clifford invariant} of a non-degenerate symmetric
bilinear form $(V,b)$ is defined in terms of its Clifford algebra.
The \linedef{Clifford algebra} $C(V,b)$ of $(V,b)$ is the quotient of
the tensor algebra $T(V) = \bigoplus_{r \geq 0} V^{\tensor r}$ by the
two-sided ideal generated by $\{v\tensor w + w\tensor v - b(v,w) \, :
\, v,w \in V\}$.  If $(V,b)$ has rank $r$, then $C(V,b) = C_0(V,b)
\oplus C_1(V,b)$ inherits the structure of a $\Z/2\Z$-graded
semisimple $F$-algebra of $F$-dimension $2^r$, see Scharlau
\cite[\S9.2]{scharlau:book}.  By the structure theory of the Clifford
algebra, $C(V,b)$ or $C_0(V,b)$ is a central simple $F$-algebra
depending on whether $(V,b)$ has even or odd rank, respectively.  The
\linedef{Clifford invariant} $c(V,b) \in \Br(F)$ is then the Brauer
class of $C(V,b)$ or $C_0(V,b)$, respectively.  Since the Clifford
algebra and its even subalgebra carry canonical involutions of the
first kind, their respective classes in the Brauer group are of order
2, see Knus \cite[\S\!\,IV.7.8]{knus:quadratic_hermitian_forms}.
While the Clifford invariant is not additive on $W(F)$, its
restriction to $I^2(F)$ gives rise to a group homomorphism
$$
e^2 : I^2(F) \to \Brtwo(F) \isom H^2(F,\muu_2) \isom
H^2(F,\muu_2^{\tensor 2}),
$$
see Knus \cite[IV Prop.\ 8.1.1]{knus:quadratic_hermitian_forms}.

Any symmetric bilinear form $(V,b)$ over a field of characteristic
$\neq 2$ can be diagonalized, i.e.\ a basis can be chosen for $V$ so
that the Gram matrix $M_b$ is diagonal.  For $a_1,\dots,a_r \in
F\mult$, we write $\qform{a_1,\dotsc,a_r}$ for the standard symmetric
bilinear form with associated diagonal Gram matrix.  For $a,b \in
F\mult$, denote by $(a,b)_F$ the (quaternion) $F$-algebra generated by
symbols $x$ and $y$ subject to the relations $x^2=a$, $y^2=b$, and
$xy=-yx$.  For example, $(-1,-1)_{\R}$ is Hamilton's ring of
quaternions.  Then the discriminant and Clifford invariant can be
conveniently calculated in terms of a diagonalization.  For $(V,b)
\isom \qform{a_1,\dotsc,a_r}$, we have
$$
d_{\pm}(V,b) = ((-1)^{r(r-1)/2}a_1\dotsm a_r) \in
F\mult/F\mult{}^2
$$ 
and
\begin{equation}
\label{eq:c_w}
c(V,b) = \alpha(r)(-1,a_1\dotsm a_r)_F + \beta(r)(-1,-1)_F + \sum_{i<j}(a_j, a_j)_F \in \Brtwo(F)
\end{equation}
where
$$
\alpha(r) =  \frac{(r-1)(r-2)}{2}, \qquad
\beta(r) =   \frac{(r+1)r(r-1)(r-2)}{24},
$$
see Lam \cite{lam:algebraic_theory_quadratic_forms}, Scharlau \cite[II.12.7]{scharlau:book} or Esnault--Kahn--Levine--Viehweg \cite[\S1]{esnault_kahn_levine_viehweg:arason}.

\section{Globalization of cohomology theories}
\label{sec:schemes}

Generalizations (what we will call \linedef{globalizations}) of the
Milnor conjecture to the context of rings and schemes have emerged
from many sources, see Parimala \cite{parimala:affine_three_folds},
%Guin \cite{guin:homologie_GL_Milnor_K-theory},
Colliot-Th\'el\`ene--Parimala
\cite{colliot-thelene_parimala:real_components}, Parimala--Sridharan
\cite{parimala_sridharan:graded_Witt}, Monnier
\cite{monnier:unramified}, Pardon \cite{pardon:filtered},
Elbaz-Vincent--M\"uller-Stach \cite{elbaz-vincent_muller-stach}, Gille
\cite{gille:graded_Gersten-Witt}, and Kerz
\cite{kerz:Gersten_conjecture_Milnor_K-theory}.  To begin with, one
must ask for appropriate globalizations of the objects in the
conjecture: Milnor $K$-theory, Galois cohomology theory, and the Witt
group with its fundamental filtration.  While there are many possible
choices of such globalizations, we will focus on two types: global and
unramified.

\subsection{Global globalization}
\label{subsec:Global_theories}

Let $F$ be a field of characteristic $\neq 2$.  
% For simplicity of
% exposition, we also assume that $F$ is perfect.  
Let $\Field_F$ (resp.\ $\Ring_F$) be the category of fields (resp.\
commutative unital rings) with an $F$-algebra structure together with
$F$-algebra homomorphisms.  Let $\Sch_F$ be the category of
$F$-schemes, and $\Sm_F$ the category of smooth $F$-schemes.  We will
denote, by the same names, the associated (large) Zariski sites.  Let
$\Ab$ (resp.\ $\Abtot$) be the category of abelian groups (resp.\
graded abelian groups), we will always consider $\Ab$ as embedded in
$\Abtot$ in degree 0.

Let $M^{\tot} : \Field_F \to \Abtot$ be a functor.  A
\linedef{globalization} of $M^{\tot}$ to rings (resp.\ schemes) is a
functor $\tilde{M}^{\tot} : \Ring_F \to \Abtot$ (resp.\ contravariant
functor $\tilde{M}^{\tot} : \Sch_F \to \Abtot$) extending $M^{\tot}$.
If $\tilde{M}^{\tot}$ is a globalization of $M^{\tot}$ to rings, then
we can define a globalization to schemes by taking the sheaf
$\unram{M}^{\tot}$ associated to the presheaf $U \mapsto
\tilde{M}^{\tot}(\Gamma(U,\OO_U))$ on $\Sch_F$ (always considered with
the Zariski topology).

\subsubsection*{``Na\"ive'' Milnor $K$-theory.}

For a commutative unital ring $R$, mimicking Milnor's tensorial
construction (with the additional relation that $a \tensor (-a) =0$,
which is automatic for fields) yields a graded ring $\KMiltot(R)$,
which should be referred to as ``na\"ive'' Milnor $K$-theory.
This already appears in Guin
\cite[\S3]{guin:homologie_GL_Milnor_K-theory} and later studied by
Elbaz-Vincent--M\"uller-Stach \cite{elbaz-vincent_muller-stach}.
Na\"ive Milnor $K$-theory has some bad properties when $R$ has small
finite residue fields, see Kerz
\cite{kerz:Milnor_K-theory_local_rings} who also provides a improved
Milnor $K$-theory repairing these defects.  
% Let $\KKMiltot$ be the associated globalization to schemes.
% Zariski sheaf associated to the
% presheaf of graded rings $U \mapsto \KMiltot(\Gamma(U,\OO_X))$ (or use
% the improved Milnor $K$-theory).  
% Then ``na\"ive'' Milnor $K$-theory of a scheme $X$ is defined as the
% graded ring of sections $\KMiltot(X) = \Gamma(X,\KKMiltot)$.  
% There is no truly global globalization of Milnor $K$-theory to general
% schemes.  Indeed, 
Thomason \cite{thomason:nonexistence} has shown that
there exists no globalization of Milnor $K$-theory to smooth schemes
which satisfies $\A^1$-homotopy invariance and has a functorial
homomorphism to algebraic $K$-theory.

\subsubsection*{\'Etale cohomology.}

\'Etale cohomology provides a natural globalization of Galois
cohomology to schemes.  We will thus consider the functor
$X \mapsto \Het^{\tot}(X,\muu_2^{\tensor\tot})$ on $\Sch_F$.

\subsubsection*{Global Witt group.}

For a scheme $X$, the \linedef{global} Witt group $W(X)$ of regular
symmetric bilinear forms introduced by Knebusch \cite{knebusch}
provides a natural globalization of the Witt group to schemes.  Other
possible globalizations are obtained from the Witt groups of
triangulated category with duality introduced by Balmer
\cite{balmer:derived_witt_groups},
\cite{balmer:triangular_witt_groups_I},
\cite{balmer:triangular_witt_groups_II}, \cite{balmer:handbook}.
These include: the \linedef{derived} Witt group of the bounded derived
category of coherent locally free $\OO_X$-modules; the
\linedef{coherent} Witt group of the bounded derived category of
quasicoherent $\OO_X$-modules with coherent cohomology (assuming $X$
has a dualizing complex, see Gille \cite[\S2.5]{gille:support},
\cite[\S2]{gille:graded_Gersten-Witt}); the \linedef{perfect} Witt
group of the derived category of perfect complexes of $\OO_X$-modules.
The global and derived Witt groups are canonically isomorphic by
Balmer \cite[Thm.\ 4.3]{balmer:triangular_witt_groups_II}.  All of the
above Witt groups are isomorphic (though not necessarily canonically)
if $X$ is assumed to be regular.

\subsubsection*{Fundamental filtration and the classical invariants.}

Globalizations of the classical invariants of quadratic forms are
defined as follows.  Let $(\EE,q)$ be a regular symmetric bilinear
form of rank $n$ on $X$.

The rank (modulo 2) of $\EE$ gives rise to a functorial
homomorphism
$$
e^0 : W(X) \to \Hom_{\text{cont}}(X,\Z/2\Z) = \Het^0(X,\Z/2\Z),
$$
whose kernel $I^1(X)$ is called the \linedef{fundamental ideal} of
$W(X)$.

The \linedef{signed discriminant form}
$(\det\EE,(-1)^{n(n-1)/2}\det q)$ gives rise to a functorial
homomorphism
$$
e^1 : I^1(X) \to \Het^1(X,\muu_2)
$$
see Knus \cite[III \S4.2]{knus:quadratic_hermitian_forms} and Milne
\cite[III \S4]{milne:etale_cohomology}.  Alternatively, the center of
the (even) Clifford $\OO_X$-algebra of $(\EE,q)$ defines a class in
$\Het^1(X,\Z/2\Z)$ called the \linedef{Arf invariant}, which coincides
with the signed discriminant under the canonical morphism
$\Het^1(X,\Z/2\Z) \to \Het^1(X,\muu_2)$ (see Knus \cite[IV Prop.\
4.6.3]{knus:quadratic_hermitian_forms} or Parimala--Srinivas
\cite[\S2.2]{parimala_srinivas:brauer_group_involution}).  Denote the
kernel of $e^1$ by $I^2(X)$, which is an ideal of $W(X)$.  Note that
$I^2(X)$ may not be the square of the ideal $I^1(X)$.
%, see Remark \ref{} % \ref{eq:warning}. 

The Clifford $\OO_X$-algebra
$\CliffAlg(\EE,q)$ gives rise to a functorial homomorphism
$$
e^2 : I^2(X) \to \Brtwo(X)
$$
called the \linedef{Clifford invariant}, see Knus--Ojanguren
\cite{knus_ojanguren:metabolic} and Parimala--Srinivas
\cite[\S2]{parimala_srinivas:brauer_group_involution}.  Denote the
kernel of $e^2$ by $I^3(X)$, which is an ideal of $W(X)$.

As Parimala--Srinivas \cite[p.\
223]{parimala_srinivas:brauer_group_involution} point out, there is no
functorial map $I^2(X) \to \Het^2(X,\muu_2)$ lifting the Clifford
invariant.  Instead, we can work with Grothendieck--Witt groups.  The
rank (modulo 2) gives rise to a functorial homomorphism
$$
ge^0 : GW(X) \to \Het^0(X,\Z/2\Z)
$$ 
with kernel denoted by $GI(X)$.  The signed discriminant gives rise to a
functorial homomorphism 
$$
ge^1: GI^1(X) \to \Het^1(X,\muu_2)
$$ 
with kernel denoted by $GI^2(X)$.  The class of the Clifford
$\OO_X$-algebra, together with it's canonical involution (via the
``involutive'' Brauer group construction of Parimala--Srinivas
\cite[\S2]{parimala_srinivas:brauer_group_involution}), gives rise to
a functorial homomorphism
$$
ge^2 : GI^2(X) \to \Het^2(X,\muu_2)
$$
also see
Knus--Parimala--Sridharan
\cite{knus_parimala_sridharan:compositions_triality}.  Denote the
kernel of $ge^2$ by $GI^3(X)$, which is an ideal of $GW(X)$.  

\begin{lemma}
\label{lem:GW}
Let $X$ be a scheme satisfying $\Brtwo(X) \isom {}_2\Het^2(X,\Gm)$. 
% regular quasiprojective scheme over a field $F$ of
% characteristic $\neq 2$.  
Then under the quotient map $GW(X) \to
W(X)$, the image of the ideal $GI^n(X)$ is precisely the ideal
$I^n(X)$, for $n \leq 3$.
\end{lemma}
\begin{proof}
For $n= 1,2$ this is a consequence of the following diagram with
exact rows and columns
$$
\xymatrix@R=15pt@C=30pt{
 & 0 \ar[d] & 0 \ar[d] &  & \\
 & K_0(X) \ar@{=}[r] \ar[d]^H &  K_0(X) \ar[d]^H &  & \\
0 \ar[r] & GI^{n}(X) \ar[d] \ar[r] & GI^{n-1}(X) \ar[r]^(.41){ge^{n-1}} \ar[d] &
\Het^{n-1}(X,\Z/2\Z) \ar@{=}[d] \ar[r] & 0 \\
0 \ar[r] & I^{n}(X) \ar[d] \ar[r] & I^{n-1}(X) \ar[d] \ar[r]^(.41){e^{n-1}} &
\Het^{n-1}(X,\Z/2\Z) \ar[d] \ar[r] & 0 \\
 & 0 & 0 & 0 & \\ 
}
$$
% $$
% \xymatrix@R=15pt@C=40pt{
%  & 0 \ar[d] & 0 \ar[d] &  & \\
%  & K_0(X) \ar@{=}[r] \ar[d]^H &  K_0(X) \ar[d]^H &  & \\
% 0 \ar[r] & GI^{2}(X) \ar[d] \ar[r] & GI^{1}(X) \ar[r]^(.45){ge^{1}} \ar[d] &
% \Het^{1}(X,\muu_2) \ar@{=}[d] \ar[r] & 0 \\
% 0 \ar[r] & I^{2}(X) \ar[d] \ar[r] & I^{1}(X) \ar[d] \ar[r]^(.45){e^{1}} &
% \Het^{1}(X,\muu_2) \ar[d] \ar[r] & 0 \\
%  & 0 & 0 & 0 & \\ 
% }
% $$
which is commutative since hyperbolic spaces have even rank and
trivial signed discriminant.  Here, $K_0(X)$ is the Grothendieck group
of locally free $\OO_X$-modules of finite rank and $H$ is the
hyperbolic map $\VV \mapsto H(\VV) = \bigl(\VV \oplus
\VV\dual,((v,f),(w,g)) \mapsto f(w)+g(x)\bigr)$.  

For $n=3$, we have the formula $ge^2(H(\VV)) = c_1(\VV,\muu_2)$ due to
Esnault--Kahn--Viehweg \cite[Prop.\ 5.5]{E-K-V} (combined with
\eqref{eq:c_w}).  Here $c_1(-,\muu_2)$ is the 1st Chern class modulo
2, defined as the first coboundary map in the long-exact sequence in
\'etale cohomology
$$
\dotsm \to \Pic(X) \mapto{2} \Pic(X) \mapto{c_1} \Het^2(X,\muu_2) \to \Het^2(X,
\Gm) \mapto{2} \Het^2(X,\Gm) \to \dotsm
$$
arising from the Kummer exact sequence
$$
1\to\muu_2\to\Gm\mapto{2}\Gm\to1,
$$
see Grothendieck \cite{grothendieck:classes_chern_representations}.
%By our hypothesis on $X$, we have $\Brtwo(X) \isom {}_2\Het^2(X,\Gm)$.
The claim then follows by a diagram chase through
$$
\xymatrix@R=15pt@C=40pt{
 & 0 \ar[d] & 0 \ar[d] & 0 \ar[d] \\
0 \ar[r] & K_0'(X) \ar[d]^H \ar[r] & K_0(X) \ar[d]^H \ar[r]^{\det} & \Pic(X)/2
\ar[d]^{c_1} \ar[r] & 0\\
0 \ar[r] & GI^{3}(X) \ar[d] \ar[r] & GI^{2}(X) \ar[d] \ar[r]^(.45){ge^2} & \Het^2(X,\muu_2) \ar[d] &\\
0 \ar[r] & I^{3}(X) \ar[d] \ar[r] & I^{2}(X) \ar[d]\ar[r]^(.45){e^2} & \Brtwo(X)\ar[d] & \\
& 0 & 0 & 0 & \\
}
$$
where the right vertical column arises from the Kummer sequence, and
$K_0'(X)$ is the subgroup of $K_0(X)$ generated by locally free
$\OO_X$-modules whose determinant is a square.  
% The commutativity of
% the right hand square follows from the fact that if $(\EE,q)$ is a
% metabolic form, then $ge^2(\EE,q)$ is represented in the involutive
% Brauer group by a split Azumaya algebra (see Knus--Ojanguren
% \cite{knus_ojanguren:metabolic}) with involution, which is thus in the
% image of the 1st Chern class map (modulo 2) by Parimala--Srinivas
% \cite[Thm.\ 1]{parimala_srinivas:brauer_group_involution}.
\end{proof}

\begin{remark}
The hypothesis that 
% $X$ is regular and quasi-projective over a field
% can be relaxed to simply that 
$\Brtwo(X) = {}_2\Het^2(X,\Gm)$
%.  More generally, this 
is satisfied if
$X$ is a quasi-compact quasi-separated scheme admitting an ample
invertible sheaf by de Jong's extension \cite{dejong:gabber} (see also
\cite[Th.\ 2.2.2.1]{lieblich:moduli_twisted_sheaves}) of a result of
Gabber \cite{gabber:brauer}.
% In particular, this holds for regular
% quasiprojective schemes over a field.
\end{remark}

The existence of global globalizations of the higher invariants (e.g.\
a globalization of the Arason invariant) remains a mystery.
Esnault--Kahn--Levine--Viehweg
\cite{esnault_kahn_levine_viehweg:arason} have shown that for a
regular symmetric bilinear form $(\EE,q)$ that represents a class in
$GI^3(X)$, the obstruction to having an Arason invariant in
$\Het^3(X,\Z/2\Z)$ is precisely the 2nd Chern class $c_2(\EE) \in
CH^2(X)/2$ in the Chow group modulo 2 (note that the invariant $c(\EE)
\in \Pic(X)/2$ of \cite{esnault_kahn_levine_viehweg:arason} is trivial
if $(\EE,q)$ represents a class in $GI^3(X)$).  They also provide
examples where this obstruction does not vanish.  
% This fact, together
% with the ad hoc nature of the ideals $I^n(X)$ and $GI^n(X)$, should be
% enough to dissuade us from focusing too much on the global case.
On the other hand, higher cohomological invariants always exist in
unramified cohomology.

\subsection{Unramified globalization}
\label{subsec:Unramified_replacements}

A functorial framework for the notion of ``unramified element'' is
established in Colliot-Th\'el\`ene \cite[\S2]{colliot:santa_barbara}.
See also the survey by Zainoulline
\cite[\S3]{zainoulline:purity_witt_group}.  Rost \cite[Rem.\
5.2]{rost:cycle_modules} gives a different perspective in terms of
cycle modules, also see Morel \cite[\S2]{morel:proof_milnor}.  Assume
that $X$ has finite Krull dimension and is equidimensional over a
field $F$.  For simplicity of exposition, assume that $X$ is integral.
Denote by $X^{(i)}$ its set of codimension $i$ points.
% For a point $\eta$ of $X$, denote
% by $\ol{\eta}$ the Zariski closure of the set $\{\eta\}$.

Denote by $\Local_F$ the category of local $F$-algebras together with
local $F$-algebra morphisms.  Given a functor $M^{\tot} : \Local_F \to
\Abtot$, call
$$ 
M_{\ur}^{\tot}(X)
= %\prod_{\eta \in X^{(0)}} \bigcap_{x \in \ol{\eta}^{(1)}}
\bigcap_{x \in X^{(1)}} \im \bigl( M^{\tot}(\OO_{X,x}) \to
M^{\tot}(F(X)) \bigr)
$$
the group of \linedef{unramified elements} of $M^{\tot}$ over $X$.
Then $X \mapsto M_{\ur}^{\tot}(X)$ is a globalization of $M^{\tot}$ to
schemes.  

Given a functor $M^{\tot} : \Sch_F \to \Abtot$, there is a natural map
$M^{\tot}(X) \to M_{\ur}^{\tot}(X)$.  If this map is injective,
surjective, or bijective we say that the \linedef{injectivity},
\linedef{weak purity}, or \linedef{purity} property hold,
respectively.  Whether these properties hold for various functors
$M^{\tot}$ and schemes $X$ is the subject of many conjectures and open
problems, see Colliot-Th\'el\`ene \cite[\S2.2]{colliot:santa_barbara}
for examples.

% Under certain
% conditions on $M$ (for instance, arising from a cycle module),
% $M_{\ur}$ is a sheaf on the large Zariski site of schemes.

% Given a functor $M^{\tot} : \Ring_F \to \Abtot$ with associated
% Zariski sheaf $\mathcal{M}$, then $\Gamma(X,\unram{M}) = M_{\ur}(X)$
% if 
% No.  Need Bloch-Ogus-Gabber to hold for M.

\subsubsection*{Unramified Milnor $K$-theory.}

Define the \linedef{unramified Milnor $K$-theory (resp.\ modulo 2)} of
$X$ to be the graded ring of unramified elements $\KMilurtot(X)$
(resp.\ $\KMilurtot/2(X)$) of the ``na\"ive'' Milnor $K$-theory
(resp.\ modulo 2) functor $\KMiltot$ (resp.\ $\KMiltot/2$) restricted
to $\Local_F$, see \S\ref{subsec:Global_theories}.  Let $\KKMiltot$ be
the Zariski sheaf on $\Sch_F$ associated to ``na\"ive'' Milnor
$K$-theory and $\KKMiltot/2$ the associate sheaf quotient, which is
the Zariski sheaf associated to the presheaf $U \mapsto \KMiltot(U)/2$, see
Morel \cite[Lemma 2.7]{morel:proof_milnor}.  Then $\KMilurtot(X) =
\Gamma(X,\KKMiltot)$ and $\KMilurtot/2(X) = \Gamma(X,\KKMiltot/2)$
when $X$ is smooth over an infinite field (compare with the remark in
\S\ref{subsec:Global_theories}) by the Bloch--Ogus--Gabber theorem for
Milnor $K$-theory, see Colliot-Th\'el\`ene--Hoobler--Kahn \cite[Cor.\
5.1.11, \S7.3(5)]{colliot-thelene_hoobler_kahn:Bloch-Ogus-Gabber}.
Also, see Kerz \cite{kerz:Gersten_conjecture_Milnor_K-theory}.
Note that the long exact sequence in Zariski cohomology yields a short
exact sequence
$$
0 \to \KMilurtot(X)/2 \to \KMilurtot/2(X) \to
{}_2H^1(X,\KKMiltot) \to 0
$$
still assuming $X$ is smooth over an infinite field.

\subsubsection*{Unramified cohomology.}

Define the \linedef{unramified \'etale cohomology (modulo 2)} of $X$
to be the graded ring of unramified elements
$\Hur^{\tot}(X,\muu_2^{\tensor\tot})$ of the
% Zariski sheaf associated to the
functor $\Het^{\tot}(-,\muu_2^{\tensor\tot})$.  Letting
$\unram{H}_{\et}^{\tot}$ be the Zariski sheaf on $\Sch_F$ associated
to the functor $\Het^{\tot}(-,\muu_2^{\tensor\tot})$, then
$\Gamma(X,\unram{H}_{\et}^{\tot}) = \Hur^{\tot}(X,\Z/2\Z)$ when $X$ is
smooth over a field of characteristic $\neq 2$ by the exactness of the
Gersten complex for \'etale cohomology, see Bloch--Ogus \cite[\S2.1,
Thm.\ 4.2]{bloch_ogus}.

\subsubsection*{Unramified fundamental filtration of the Witt group.}

Define the \linedef{unramified Witt group} of $X$ to be the abelian
group of unramified elements $W_{\ur}(X)$ of the global Witt group
functor $W$.  Letting $\unram{W}$ be the Zariski sheaf associated to
the global Witt group functor, then $W_{\ur}(X) = \Gamma(X,\unram{W})$
when $X$ is regular over a field of characteristic $\neq 2$ by
Ojanguren--Panin \cite{ojanguren_panin:purity} (also see Morel
\cite[Thm.\ 2.2]{morel:proof_milnor}).  Writing $I_{\ur}^{n}(X) =
I^{n}(F(X)) \cap W_{\ur}(X)$, then the functors $I_{\ur}^{n}(-)$ are
Zariski sheaves (still assuming $X$ is regular), denoted by
$\unram{I}^{n}$, which form a filtration of $\unram{W}$, see Morel
\cite[Thm.\ 2.3]{morel:proof_milnor}.

% The Zariski sheaf associated to the presheaf $U \mapsto
% I_{\ur}^{\tot}(U)/I_{\ur}^{\tot+1}(U)$ is isomorphic to the quotient
% sheaf $\unram{I}^{\tot}/\unram{I}^{\tot+1}$ by Morel
% \cite[\S2.4]{morel:proof_milnor}.  
Note that the long exact sequence in Zariski cohomology yields short
exact sequences
$$
0 \to \Iur^{n}(X)/\Iur^{n+1}(X) \to
\unram{I}^{n}/\unram{I}^{n+1}(X) \to H^1(X,\unram{I}^{n+1})' \to 0 
$$
where $H^1(X,\unram{I}^{n+1})' = \ker\bigl( H^1(X,\unram{I}^n) \to
H^1(X,\unram{I}^{n+1}) \bigr)$ and we are still assuming $X$ is
regular over a field of characteristic $\neq 2$.  If the obstruction
group $H^1(X,\unram{I}^{n+!})'$ is nontrivial, then not every element
of $\unram{I}^{n}/\unram{I}^{n+1}(X)$ is represented by a quadratic
form on $X$.  If $X$ is the spectrum of a regular local ring, then
$I_{\ur}^{\tot}(X)/I_{\ur}^{\tot+1}(X) =
\unram{I}^{\tot}/\unram{I}^{\tot+1}(X)$, see Morel \cite[Thm.\
2.12]{morel:proof_milnor}.

\begin{remark}
\label{rem:warning}
As before, the notation $I_{\ur}^n(X)$ does not necessarily mean the
$n$th power of $I_{\ur}(X)$.  This is true, however, when $X$ is the
spectrum of a regular local ring containing an infinite field of
characteristic $\neq 2$, see Kerz--M\"uller-Stach \cite[Cor.\
0.5]{kerz:Milnor-Chow_homomorphism}.  
\end{remark}

% Note that for the spectrum of a regular local ring, the ideals $I^n(X)$ can be
% defined, see \cite{}.

\subsection{Gersten complexes}

Gersten complexes (Cousin complexes) exists in a very general
framework, but for our purposes, we will only need the Gersten complex
for Milnor $K$-theory, \'etale cohomology, and Witt groups.

\subsubsection*{Gersten complex for Milnor $K$-theory.}

Let $X$ be a regular excellent integral $F$-scheme.
% over an infinite field.  
Let $C(X,\KMil^n)$ denote the Gersten complex for Milnor $K$-theory
$$
\def\objectstyle{\scriptstyle} \def\labelstyle{\scriptstyle}
\xymatrix@C=22pt{
0 \ar[r] & 
  \KMil^n(F(X)) \ar[r]^(.4){\res^{\KMil^n}} &
  \smash{\underset{{x \in X^{(1)}}}{\textstyle\bigoplus}} \KMil^{n-1}(F(x)) \ar[r]^(.54){\res^{\KMil^{n-1}}}&
  \smash{\underset{y \in X^{(2)}}{\textstyle\bigoplus}} \KMil^{n-2}(F(y)) \ar[r] & \dotsm 
}
$$
% $$
% \scriptstyle
% %0 \to \bigoplus_{\eta \in X^{(0)}} \KMil^n(\eta) 
% 0 \to \KMil^n(F(X))
%   \mapto{\res^{\KMil^n}} \bigoplus_{x \in X^{(1)}} \KMil^{n-1}(F(x))
%   \mapto{\res^{\KMil^{n-1}}} \bigoplus_{y \in X^{(2)}} \KMil^{n-2}(F(y)) \to \dotsm 
% $$
where 
% $\KMil^n(x)$ denotes the Milnor $K$-group of the residue field
% $F(x)$ for a point $x$ of $X$ and 
$\res^{\KMil}$ is the ``tame symbol'' homomorphism defined in Milnor
\cite[Lemma 2.1]{milnor:conjecture}.
% Describe tame symbol.  
%Then $C(X,\KMil^n)$ is a resolution of $\KMilur^n}(X)$.  
We have that $H^0(C(X,\KMil^n)) = \KMilur^n(X)$.  See Rost
\cite[\S1]{rost:cycle_modules} or Fasel \cite[Ch.\ 2]{fasel:Chow-Witt}
for more details.  We will also consider the Gersten complex
$(C,\KMil^n/2)$ for Milnor $K$-theory modulo 2, for which we have that
$H^0(C(X,\KMil^n/2))=\KMilur^n/2(X)$.

\subsubsection*{Gersten complex for \'etale cohomology.}

Let $X$ be a smooth integral $F$-scheme, with $F$ of characteristic
$\neq 2$.  Let $C(X,H^n)$ denote the Gersten complex for \'etale
cohomology
$$
\def\objectstyle{\scriptstyle} \def\labelstyle{\scriptstyle}
\xymatrix@C=22pt{
0 \ar[r] & H^{n}(F(X)) \ar[r]^(.4){\res^{H^n}} 
& \smash{\underset{{x \in X^{(1)}}}{\textstyle\bigoplus}} H^{n-1}(F(x))
  \ar[r]^(.55){\res^{H^{n-1}}} 
& \smash{\underset{{y \in X^{(2)}}}{\textstyle\bigoplus}}
H^{n-2}(F(y)) \ar[r] & \dotsm 
}
$$
% $$
% %0 \to \bigoplus_{\eta \in X^{(0)}} H^n(\eta,\Z/2\Z) 
% \xymatrix@R=15pt@C=10pt{
% 0 \ar[r] & H^n(F(X),\muu_2^{\tensor n}) \ar[r]%^{\res^{H^n}} 
% & \bigoplus_{x \in X^{(1)}} H^{n-1}(F(x),\muu_2^{\tensor (n-1)})
%   \ar[r]%^{\res^{H^{n-1}}} 
% & \bigoplus_{y \in X^{(2)}} H^{n-2}(F(y),\muu_2^{\tensor (n-2)}) \to \dotsm 
% }
% $$
where 
$H^n(-) = H^{n}(-,\muu_2^{\tensor n})$ and
% $H^n(x,\muu_2^{\tensor n})$ denotes the Galois cohomology group
% of the residue field $F(x)$ for a point $x$ of $X$ and 
$\res^H$ is the homomorphism induced from the spectral sequence
associated to the coniveau filtration, see Bloch--Ogus
\cite{bloch_ogus}.  Then we have that $C(X,H^n)$ is a resolution of
$\Hur^n(X,\muu_2^{\tensor n})$.
% Describe residue maps.  

\subsubsection*{Gersten complex for Witt groups.}

Let $X$ be a regular
% separated noetherian
integral $F$-scheme of
finite Krull dimension.  Let $C(X,W)$ denote the Gersten--Witt complex
$$
\def\objectstyle{\scriptstyle} \def\labelstyle{\scriptstyle}
\xymatrix@C=22pt{
0 \ar[r] & W(F(X)) \ar[r]^(.44){\res^W} & 
\smash{\underset{{x \in X^{(1)}}}{\textstyle\bigoplus}} W(F(x))
  \ar[r]^{\res^W} & 
\smash{\underset{{y \in X^{(2)}}}{\textstyle\bigoplus}}
W(F(y)) \ar[r] & \dotsm 
}
$$
% $$
% % 0 \to \bigoplus_{\eta \in X^{(0)}} W(\eta) 
% 0 \to W(F(X))
%   \mapto{\res^W} \bigoplus_{x \in X^{(1)}} W(F(x))
%   \mapto{\res^W} \bigoplus_{y \in X^{(2)}} W(F(y)) \to \dotsm 
% $$
where 
% $W(x)$ denotes the Witt group of the residue field $F(x)$ for a
% point $x$ of $X$ and 
$\res^W$ is the homomorphism induced from the second residue map for a
set of choices of local parameters, see Balmer--Walter
\cite{balmer_walter:Gersten-Witt}.  Because of these choices, $C(X,W)$
is only defined up to isomorphism, though there is a canonical complex
defined in terms of Witt groups of finite length modules over the
local rings of points.
% Describe residue maps.  
We have that $H^0(C(X,W)) = W_{\ur}(X)$.

\subsubsection*{Fundamental filtration.}

The filtration of the Gersten complex for Witt groups induced by the
fundamental filtration was first studied methodically by
Arason--Elman--Jacob \cite{arason_elman_jacob:graded_Witt_I}, see also
Parimala--Sridharan \cite{parimala_sridharan:graded_Witt}, Gille \cite{gille:graded_Gersten-Witt}, and Fasel
\cite[\S9]{fasel:Chow-Witt}.

The differentials of the Gersten complex for Witt groups respect the
fundamental filtration as follows:
$$
\res^{I^n} \Bigl( \bigoplus_{x \in X^{(p)}} I^n(F(x)) \Bigr) \subset \bigoplus_{y \in X^{(p+1)}} I^{n-1}(F(y)),
$$ 
see Fasel \cite[Thm.\ 9.2.4]{fasel:Chow-Witt} and Gille
\cite{gille:graded_Gersten-Witt}.  Thus for all $n \geq 0$ we have
complexes $C(X,I^n)$
$$
\def\objectstyle{\scriptstyle} \def\labelstyle{\scriptstyle}
\xymatrix@C=22pt{
0 \ar[r] & I^{n}(F(X)) \ar[r]^(.4){\res^{I^n}} 
& \smash{\underset{{x \in X^{(1)}}}{\textstyle\bigoplus}} I^{n-1}(F(x))
  \ar[r]^(.55){\res^{I^{n-1}}} 
& \smash{\underset{{y \in X^{(2)}}}{\textstyle\bigoplus}}
I^{n-2}(F(y)) \ar[r] & \dotsm 
}
$$
% $$
% %0 \to \bigoplus_{\eta \in X^{(0)}} I^n(\eta) 
% 0 \to I^n(F(X)) \mapto{\res^{I^n}} \bigoplus_{x \in X^{(1)}}
% I^{n-1}(F(x)) \mapto{\res^{I^{n-1}}} \bigoplus_{y \in X^{(2)}}
% I^{n-2}(F(y)) \to \dotsm
% $$
which provide a filtration of $C(X,W)$ in the category of complexes of
abelian groups.  Here we write $I^n(-) = W(-)$ for $n \leq 0$.  We
have that $H^0(C(X,I^n)) = I_{\ur}^n(X)$.

The canonical inclusion $C(X,I^{n+1}) \to C(X,I^{n})$ respects the
differentials, and so defines a cokernel complex $C(X,I^n/I^{n+1})$
$$
\def\objectstyle{\scriptstyle} \def\labelstyle{\scriptstyle}
\xymatrix@C=16pt{
0 \ar[r] & I^n/I^{n+1}(F(X)) \ar[r]%^(.4){\res^{I^n/I^{n+1}}} 
&\smash{\underset{{x \in X^{(1)}}}{\textstyle\bigoplus}}
I^{n-1}/I^n(F(x)) \ar[r]%^{\res^{I^{n-1}/I^n}} 
&\smash{\underset{{y \in X^{(2)}}}{\textstyle\bigoplus}}
I^{n-2}/I^{n-1}(F(y)) \ar[r] & \dotsm
}
$$
% $$
% \xymatrix@R=15pt@C=10pt{
% 0 \ar[r] & I^n/I^{n+1}(F(X)) \ar[r]%^{\res^{I^n/I^{n+1}}} 
% &\bigoplus_{x \in X^{(1)}}
% I^{n-1}/I^n(F(x)) \ar[r]%^{\res^{I^{n-1}/I^n}} 
% &\bigoplus_{y \in X^{(2)}}
% I^{n-2}/I^{n-1}(F(y)) \to \dotsm
% }
% $$
see Fasel \cite[D\'ef.\ 9.2.10]{fasel:Chow-Witt}, where
$I^n/I^{n+1}(L) = I^n(L)/I^{n+1}(L)$ for a field $L$.  
We have that $H^0(C(X,I^n/I^{n+1}))=\unram{I}^n/\unram{I}^{n+1}(X)$
% In general,
% there is only an inclusion on the level of cohomology
% $I_{\ur}^n/I_{\ur}^{n+1}(X) \subsetto
% H^0(C(X,I^n/I^{n+1}))$.%, see Remark \ref{}.%\ref{rem:wrong}.

\subsubsection*{Unramified norm residue symbol.}

The norm residue symbol for fields provides a morphism of complexes
$h^n : C(X,\KMil^n/2) \to C(X,H^n)$, where the map on terms of degree
$j$
%$C^j(X,\KMil^n/2)\to C^j(X,H^n)$ 
is $h^{n-j}$.  By the Milnor conjecture for fields, this is an
isomorphism of complexes.  Upon restriction, we have an isomorphism
$h_{\ur}^n : \KMilur^n/2(X) \to \Hur^n(X,\muu_2^{\tensor n})$.  Upon
further restriction, we have an injection $h_{\ur}^n : \KMilur^n(X)/2 \to \Hur^n(X,\muu_2^{\tensor n})$.

\subsubsection*{Unramified Pfister form map.}

The Pfister form map for fields provides a morphism of
complexes $s^n : C(X,\KMil^n/2) \to C(X,I^n/I^{n+1})$, where the map
%$C^j(X,\KMil^n/2) \to C^j(X,I^n/I^{n+1})$ 
on terms of degree $j$ is $s^{n-j}$.  Upon restriction, we have a
homomorphism $s_{\ur}^n : \KMilur^n/2(X) \to
\unram{I}^n/\unram{I}^{n+1}(X)$.  See Fasel \cite[Thm.\
10.2.6]{fasel:Chow-Witt}.

\subsubsection*{Unramified higher cohomological invariants.}

By the Milnor conjecture for fields, there exists a higher
cohomological invariant morphism of complexes $e^n : C(X,I^n) \to
C(X,H^n)$, where the map 
%$C^j(X,I^n) \to C^j(X,H^n)$
on terms of degree $j$ is $e^{n-j}$.  Upon
restriction, we have homomorphisms $e_{\ur}^{n} : I_{\ur}^n(X) \to
\Hur^n(X,\muu_2^{\tensor n})$ factoring through to $e_{\ur}^{n} :
I_{\ur}^n(X)/I_{\ur}^{n+1}(X) \to \Hur^n(X,\muu_2^{\tensor n})$. 

Furthermore, on the level of complexes, the higher cohomological
invariant morphisms factors through to a morphism of complexes $e^n :
C(X,I^n/I^{n+1}) \to C(X,H^n)$, which by the Milnor conjecture over
fields, is an isomorphism.  Upon restriction, we have isomorphisms
$e_{\ur}^n : \unram{I}^n/\unram{I}^{n+1}(X) \to
\Hur^n(X,\muu_2^{\tensor n})$.  Also see Morel
\cite[\S2.3]{morel:proof_milnor}.

\subsection{Motivic globalization}
\label{motivic}

There is another important globalization of Milnor $K$-theory and
Galois cohomology, but we only briefly mention it here.  Conjectured
to exist by Be\u{\i}linson \cite{beilinson:pairings} and Lichtenbaum
\cite{lichtenbaum:values}, and then constructed by Voevodsky
\cite{voevodsky:Milnor_conjecture_I}, motivic complexes modulo 2 give
rise to Zariski and \'etale motivic cohomology groups modulo 2
$H_{\Zar}^{n}(X,\Z/2\Z(m))$ and $H_{\et}^{n}(X,\Z/2\Z(m))$.

For a field $F$, Nesterenko--Suslin \cite{nesterenko_suslin} and
Totaro \cite{totaro:Milnor_K-theory} establish a canonical isomorphism
$H_{\Zar}^{n}(\Spec F,\Z/2\Z(n)) \isom \KMil^2(F)/2$ while the work of
Bloch, Gabber, and Suslin (see the survey by Geisser
\cite[\S1.3.1]{geisser:K-theory_handbook}) establishes an isomorphism
$H_{\et}^{n}(\Spec F,\Z/2\Z(n)) \isom H^n(F,\muu_2^{\tensor n})$ (for
$F$ of characteristic $\neq 2$).  The natural pullback map
$$
\varepsilon^* : H_{\Zar}^{n}(\Spec F,\Z/2\Z(n)) \to H_{\et}^{n}(\Spec
F,\Z/2\Z(n))
$$ 
induced from the change of site $\varepsilon : X_{\et} \to X_{\Zar}$
is then identified with the norm residue homomorphism.  Thus
$H_{\Zar}^{n}(-,\Z/2\Z(n))$ and $\Het^{n}(-,\Z/2\Z(n))$ provide
\linedef{motivic} globalizations of the mod 2 Milnor $K$-theory and
Galois cohomology functors, respectively.  On the other hand, there
does not seem to exist a direct motivic globalization of the Witt
group or its fundamental filtration.

% Restricting the norm residue symbol and Pfister form map on the
% Gersten complexes yields well defined homomorphisms

% Mention Gersten complex and Gersten spectral sequence and gradings,
% from Gille \cite{gille:graded_Gersten-Witt}...

\section{Globalization of the Milnor conjecture}
\label{subsec:Unramified_Milnor_conjecture}

\subsubsection*{Unramified.}
Let $F$ be a field of characteristic $\neq 2$.  Summarizing the
results cited in \S2.2--2.3, we have a commutative diagram of isomorphisms
\begin{equation*}
\xymatrix{
\KKMil^{\tot}/2 \ar[r]^{h^{\tot}} \ar[d]_{s^{\tot}}& \unram{H}_{\et}^{\tot} \\
\unram{I}^{\tot}/\unram{I}^{\tot+1} \ar[ur]_{e^{\tot}}& 
}
\end{equation*} 
of sheaves of graded abelian groups on $\Sm_F$.  In particular, we
have such a commutative diagram of isomorphisms on the level of global
sections.  What we will consider as a globalization of the Milnor
conjecture --- the \linedef{unramified Milnor question} --- is a
refinement of this. 

\begin{itemize}
\item[(S)] Let $X$ be a smooth scheme over a field of characteristic $\neq 2$.
Then restricting $s_{\ur}^{\tot}$ gives rise to a homomorphism
$s_{\ur}^{\tot} : \KMilurtot(X)/2 \to
\Iur^{\tot}(X)/\Iur^{\tot+1}(X)$.
\end{itemize}

\begin{question}[Unramified Milnor question]
\label{question}
Let $X$ be a smooth scheme over a field of characteristic $\neq 2$.
Consider the following diagram:
$$
\xymatrix{
\KMilurtot(X)/2 \ar@{-->}[d]_{?} \ar@{^{(}->}[r]^(.54){i_K^{\tot}} & \KKMil^{\tot}/2(X) \ar[r]^{h_{\ur}^{\tot}} \ar[d]_{s_{\ur}^{\tot}}& \Hur^{\tot}(X,\muu_2^{\tensor\tot}) \\
I_{\ur}^{\tot}(X)/I_{\ur}^{\tot+1}(X) \ar@{^{(}->}[r]^(.54){i_I^{\tot}} &
\unram{I}^{n}/\unram{I}^{n+1}(X) \ar[ur]_{e_{\ur}^{\tot}}& 
}
$$
\begin{enumerate}
\item Is the inclusion $i_I^{\tot}$ surjective?

\item Is the inclusion $i_K^{\tot}$ surjective? 

\item Does the restriction of $s_{\ur}^{\tot}$ to $\KMilurtot(X)/2$
have image contained in $I_{\ur}^{\tot}(X)/I_{\ur}^{\tot+1}(X)$?  If
so, is it an isomorphism?
\end{enumerate}
\end{question}

% If Question \ref{question} (1) (resp.\ (2)) has a positive answer over
% a given scheme $X$, then we say that ``the unramified Milnor
% conjecture for quadratic forms (resp.\ Milnor $K$-theory) holds over
% $X$.''  
Note that in degree $n$, Questions \ref{question} (1), (2), and (3)
can be rephrased in terms of the obstruction groups, respectively:
does $H^1(X,\unram{I}^{n+1})'$ vanish; does ${}_2 H^1(X,\KKMil^{n})$
vanish; and does the restriction of $s_{\ur}^n$ yield a map ${}_2
H^1(X,\KKMil^{n}) \to H^1(X,\unram{I}^{n+1})'$ and is it an
isomorphism?

From now on we shall focus mainly on the unramified Milnor question
for quadratic forms (i.e.\ Question \ref{question}(1)), which was already explicitly asked by
Parimala--Sridharan \cite[Question Q]{parimala_sridharan:graded_Witt}.

\subsubsection*{Global Grothendieck--Witt.}
We mention a global globalization of the Milnor conjecture for
quadratic forms.  Because of the conditional definition of the global
cohomological invariants, we restrict ourselves to the classical
invariants on Grothendieck--Witt groups defined in
\S\ref{subsec:Global_theories}.

\begin{question}[Global Merkurjev question]
\label{question:GW}
Let $X$ be a regular scheme with 2 invertible.  For $n \leq 2$,
consider the homomorphisms,
$$
ge^n : GI^n(X)/GI^{n+1}(X) \to \Het^n(X,\Z/2\Z)
$$
induced from the (classical) cohomological invariants on
Grothendieck--Witt groups.  Are they surjective?
\end{question}

This can be viewed as a globalization of Merkurjev's theorem. Indeed,
first note that the cases $n=0,1$ of Question \ref{question:GW} are
easy.  Next, a consequence of Lemma \ref{lem:GW} is that $ge^2 :
GI^2(X) \to \Het^2(X,\Z/2\Z)$ is surjective if and only $e^2 : I^2(X)
\to \Brtwo(X)$ is surjective. This, in turn, is a consequence of a
positive answer to Question \ref{question}(1) for any $X$ satisfying
weak purity for the Witt group (see \S\ref{subsec:purity} for
examples).

\subsubsection*{Motivic.}

The globalization of the Milnor conjecture for Milnor $K$-theory using
Zariski and \'etale motivic cohomology modulo 2 (see \S\ref{motivic})
is the $(n,n)$ modulo 2 case of the Be\u{\i}linson--Lichtenbaum
conjecture: for a smooth variety $X$ over a field, the canonical map
$H_{\Zar}^{n}(X,\Z/2\Z(m)) \to H_{\et}^{n}(X,\Z/2\Z(m))$ is an
isomorphism for $n \leq m$.  The combined work of Suslin--Voevodsky
\cite{suslin_voevodsky} and Geisser--Levine
\cite{geisser_levine:bloch-kato} show the Be\u{\i}linson--Lichtenbaum
conjecture to be a consequence of the Bloch--Kato conjecture, a proof
of which has been announced by Rost and Voevodsky.

\subsection{Some purity results}
\label{subsec:purity}

In this section we review some of the purity results (see
\S\ref{subsec:Unramified_replacements}) relating the global and
unramified Witt groups and cohomology.

\subsubsection*{Purity results for Witt groups.}

For a survey on purity results for Witt groups, see Zainoulline
\cite{zainoulline:purity_witt_group}.  Purity for the global Witt
group means that the natural map $W(X) \to W_{\ur}(X)$ is an
isomorphism.

\begin{theorem}
\label{purity_witt}
Let $X$ be a regular integral noetherian scheme with 2 invertible.
Then purity holds for the global Witt group functor under the
following hypotheses:
\begin{enumerate}
\item $X$ is dimension $\leq 3$,
\item $X$ is the spectrum of a regular local ring of dimension $\leq 4$,
\item $X$ is the spectrum of a regular local ring containing a field.
\end{enumerate}
\end{theorem}

For part (1), the case of dimension $\leq 2$ is due to
Colliot-Th\'el\`ene--Sansuc \cite[Cor.\
2.5]{colliot-thelene_sansuc:fibres_quadratiques}, the case of
dimension $3$ and $X$ affine is due to
Ojanguren--Parimala--Sridharan--Suresh
\cite{ojanguren_parimala_sridharan_suresh:purity_threefolds}, and for
the general case (as well as (2)) see Balmer--Walter
\cite{balmer_walter:Gersten-Witt}.  For (3), see Ojanguren--Panin
\cite{ojanguren_panin:purity}.

As a consequence, for any scheme over which purity for the Witt group
holds, the unramified Milnor question for $e^n$
% and $ge^{\tot}$ 
(with $n\leq 2$) is equivalent to the analogous global Milnor
question.  This is especially useful for the case of curves.
% For any affine scheme over which
% purity for the Witt group holds, the global and unramified Milnor
% conjectures for $e^{\tot}$ are equivalent.

\subsubsection*{Purity results for \'etale cohomology.}

For $X$ geometrically locally factorial and integral, the purity
property holds for unramified cohomology in degree $\leq 1$, i.e.\
$$
\Het^0(X,\Z/2\Z) = \Hur^0(X,\Z/2\Z) = \Z/2\Z, \quad \text{and} \quad
\Het^1(X,\muu_2) = \Hur^1(X,\muu_2)
$$ 
see Colliot-Th\'el\`ene--Sansuc \cite[Cor.\ 3.2, Prop.\
4.1]{colliot-thelene_sansuc:type_multiplicatif}.
%  (note that $\HZar^1(X,\muu_2)=0$ since $\muu_2$ is flasque in the
%  Zariski topology).  

For $X$ smooth over a field of characteristic $\neq 2$, we have a
canonical identification $\Brtwo(X) = \Hur^2(X,\muu_2)$ by Bloch--Ogus
\cite{bloch_ogus} such that the canonical map $\Het^2(X,\muu_2) \to
\Hur^2(X,\muu_2) = \Brtwo(X)$ is the map arising from the Kummer exact
sequence already considered in the proof of Lemma \ref{lem:GW}.

% Let $X$ be a noetherian regular integral scheme of dimension 1, then
% for $\tot=2$ we have $\Brtwo(X) = \Hur^2(X,\Z/2\Z)$, see
% Grothendieck \cite[Prop.\ 2.1]{grothendieck:Brauer_III}, and
% $\Het^2(X,\Z/2\Z) \to \Brtwo(X)$ is the canonical forgetful map
% arising from the Kummer exact sequence.

\subsection{Positive results}
\label{subsec:Results}

We now survey some of the known positive cases of the unramified
Milnor question in the literature.

\begin{theorem}[{Kerz--M\"{u}ller-Stach \cite[Cor.\ 0.8]{kerz:Milnor-Chow_homomorphism}, Kerz \cite[Thm.\ 1.2]{kerz:Gersten_conjecture_Milnor_K-theory}}]
Let $R$ be a local ring with infinite residue field of characteristic
$\neq 2$.  Then the unramified Milnor question (all parts of Question
\ref{question}) has a positive answer over $\Spec R$.
\end{theorem}

Hoobler \cite{hoobler:Merkurjev-Suslin_semilocal_ring} had already
proved this in degree $2$.  
% Since purity holds for such local rings,
% the unramified Milnor question is equivalent to its global
% counterpart,.

% \begin{theorem}[{Kerz \cite{kerz:Gersten_conjecture_Milnor_K-theory}}]
% Let $X$ be a smooth scheme over an infinite field $F$.  Then the
% unramified Milnor conjecture for $h_{\ur}^{\tot}$ holds over $X$.
% \end{theorem}

The following result was communicated to us by Stefan Gille (who was
inspired by Totaro \cite{totaro:non-injectivity}).

\begin{theorem}
Let $X$ be a proper smooth integral variety over a field of
characteristic $\neq 2$.  If $X$ is $F$-rational then the unramified
Milnor question (all parts of Question \ref{question}) has a positive
answer over $X$.
\end{theorem}
\begin{proof}
The groups $K_{\mathrm{M},\ur}^n(X)$, $\Hur^n(X,\muu_2^{\tensor n})$,
and $I_{\ur}^n(X)$ are birational invariants of smooth proper
$F$-varieties.  To see this, one can use Colliot-Th\'el\`ene
\cite[Prop.\ 2.1.8e]{colliot:santa_barbara} and the fact that the
these functors satisfy weak purity for regular local rings (see
Theorem \ref{purity_witt}).  Another
proof uses the fact that the complexes $C(X,\KMil^n)$, $C(X,H^n)$, and
$C(X,I^n)$ are cycle modules in the sense of Rost, see \cite[Cor.\
12.10]{rost:cycle_modules}.  In any case, by Colliot-Th\'el\`ene
\cite[Prop.\ 2.1.9]{colliot:santa_barbara} the pullback induces
isomorphisms $\KMil^{n}(F) \isom K_{\mathrm{M},\ur}^{n}(\P^m)$ (first
proved by Milnor \cite[Thm.\ 2.3]{milnor:conjecture} for $\P^1$),
$H^n(F,\muu_2^{\tensor n}) \isom \Hur^n(\P^m,\muu_2^{\tensor n})$, and
$I_{\ur}^n(F) \isom I_{\ur}^n(F)(\P^m)$ for all $n \geq 0$ and $m \geq
1$.  In particular, $\KMil^{n}(F)/2 \isom \KMilur^n(X)/2$ and
$I_{\ur}^n(F)/\Iur^{n+1}(F) \isom \Iur^n(X)/\Iur^{n+1}(X)$, and the
theorem follows from the Milnor conjecture over fields.
\end{proof}

The following positive results are known for low dimensional schemes.
Recall the notion of \linedef{cohomological dimension} $\cd(F)$ of a
field (see Serre \cite[I \S3.1]{serre:cohomologie_galoisienne}),
\linedef{virtual cohomological 2-dimension} $\vcd_2(F) = \cd_2(
F(\sqrt{-1}))$ and their 2-primary versions.  Denoting by $d(F)$ any
of these notions of dimension, note that if $d(F) \leq k$ and $\dim X
\leq l$ then $d (F(X)) \leq k+l$.  
% For $n > \cd_2(F(X))$, we have
% $\unram{H}^n=0$ (when restricted to the small Zariski site of $X$) and
% hence $\unram{I}^n=\unram{I}^{n+1}$ and $\KKMil^n=2\KKMil^n$; in
% particular, Questions \ref{question}(1), (2), and (3) have
% positive answers over $X$ in degree $> \cd_2(F(X))$.  Also,
% $\unram{I}^n \isom \unram{H}^n$ for $n=\cd_2(F(X))$, see
% Parimala--Sridharan \cite[Lemma 4.1]{parimala_sridharan:graded_Witt}. 

% \begin{theorem}[{Parimala--Sridharan
% \cite[Lemma 4.2, Thm.\ 6.1]{parimala_sridharan:graded_Witt}}]
% Let $X$ be a smooth integral affine curve over a field $F$ of
% characteristic $\neq 2$.  Then the unramified Milnor conjecture holds
% when $\cd_2 F \leq 2$ or $F$ is a global field.
% \end{theorem}

% \begin{theorem}
% Let $X$ be a smooth integral proper curve over a field $F$ of
% characteristic $\neq 2$.  The the unramified Milnor conjecture holds
% when 
% \end{theorem}

\begin{theorem}[{Parimala--Sridharan
\cite{parimala_sridharan:graded_Witt}, Monnier \cite{monnier:unramified}}]
\label{thm:curves}
Let $X$ be a smooth integral curve over a field $F$ of characteristic
$\neq 2$.  Then the unramified Milnor question for quadratic forms
(Question \ref{question}(1)) has a positive answer over $X$ in the
following cases:
\begin{enumerate}
% \item $F$ is quadratically closed, i.e.\ $\cd_2 F=0$,
% \item $F$ is real closed, i.e.\ $\vcd F = 0$,
\item $\cd_2(F) \leq 1$,

\item $\vcd (F) \leq 1$, 

\item $\cd_2 (F) = 2$ and $X$ is affine,

\item $\vcd (F) = 2$ and $X$ is affine.
\end{enumerate}
\end{theorem}
\begin{proof}
For (1), this
% the case $\cd_2 F = 0$ (i.e.\ $F$ is quadratically closed)
% follows from a variant of Tsen's theorem, while the case $\cd_2 F = 1$
follows from Parimala--Sridharan \cite[Lemma
4.1]{parimala_sridharan:graded_Witt} and the fact that $e^1$ is always
surjective.  For (2), the case $\vcd(F)=0$ (i.e.\ $F$ is real closed)
is contained in Monnier \cite[Cor.\ 3.2]{monnier:unramified} and the
case $\vcd(F) = 1$ follows from a straightforward generalization to
real closed fields of the results in
\cite[\S5]{parimala_sridharan:graded_Witt} for the real numbers.  For
(3), see \cite[Lemma 4.2]{parimala_sridharan:graded_Witt}.  For (4),
the statement follows from a generalization of \cite[Thm.\
6.1]{parimala_sridharan:graded_Witt}.
\end{proof}

We wonder whether $\vcd$ can be replaced by $\vcd_2$ in Theorem
\ref{thm:curves}.  Parimala--Sridharan \cite[Rem.\
4]{parimala_sridharan:graded_Witt} ask whether there exist affine
curves (over a well-chosen field) over which the unramified Milnor
question has a negative answer.

For surfaces, there are positive results are in the case of $\vcd(F)
=0$.  If $F$ is algebraically closed, then the unramified Milnor
question for quadratic forms (Question \ref{question}(1)) has a
positive answer by a direct computation, see Fern\'andez-Carmena
\cite{fernandez-carmena:Witt_group_surfaces}.  If $F$ is real closed,
one has the following result.

\begin{theorem}[{Monnier \cite[Thm.\ 4.5]{monnier:unramified}}]
Let $X$ be smooth integral surface over a real closed field $F$.  If
the number of connected components of $X(F)$ is $\leq 1$ (i.e.\ in
particular if $X(F)=\emptyset$), then the unramified Milnor question
for quadratic forms (Question \ref{question}(1)) has a positive answer
over $X$.
\end{theorem}

Examples of surfaces with many connected components over a real closed
field, and over which the unramified Milnor question still has a
positive answer, are also given in Monnier \cite{monnier:unramified}.

Finally, as a consequence of \cite[Cor.\ 3.4]{auel:surjectivity}, the
unramified Milnor question for quadratic forms (Question
\ref{question}(1)) has a positive answer over any scheme $X$
satisfying: $\Brtwo(X)$ is generated by quaternion Azumaya algebras
(i.e.\ index$|$period for 2-torsion classes); or $\Brtwo(X)$ is
generated by Azumaya algebras of degree dividing 4 (i.e.\
index$|$period$^2$ for 2-torsion classes) and $\Pic(X)$ is
2-divisible.  In particular, this recovers the known cases of curves over
finite fields (by class field theory) and surfaces over algebraically
closed fields (by de Jong \cite{dejong:surfaces}).

\section{Negative results}
\label{sec:Counterexample_to_a_global_Milnor_conjecture}

% Theorems \ref{thm:unramified_Milnor_conj} and
% \ref{thm:unramified_Milnor_conj_coherent} imply the existence
% of unramified higher cohomological invariants inducing a graded
% isomorphism $e_{\ur}^{\tot} : I_{\ur}^{\tot}/I_{\ur}^{\tot+1}(X) \to
% \Hur^{\tot}(X,\Z/\Z)$.  

Alex Hahn asked if there exists a ring $R$ over which the global
Merkurjev question (Question \ref{question:GW}) has a negative answer,
i.e.\ $e^2 : I^2(R) \to \Brtwo(R)$ is not surjective.  The results of
Parimala, Scharlau, and Sridharan \cite{parimala_scharlau:extension},
\cite{parimala_sridharan:graded_Witt},
\cite{parimala_sridharan:nonsurjectivity}, show that there exist
smooth complete curves $X$ (over $p$-adic fields $F$) over which the
unramified Milnor question (Question \ref{question}(1)) in degree 2
(and hence, by purity, the global Merkurjev question) has a negative
answer.
% such that
% $e_{\ur}^{2} : I_{\ur}^2(X) \to \Hur^2(X,\muu_2^{\tensor 2})$ is
% \emph{not surjective}.  
% Parimala--Sridharan \cite{parimala_sridharan:nonsurjectivity} finally
% use these examples to answer Alex Hahn's question.

\begin{remark}
\label{rem:wrong}
The assertion (in Gille \cite[\S10.7]{gille:graded_Gersten-Witt} and
Pardon \cite[\S5]{pardon:filtered}) that the unramified Milnor
question (Question \ref{question}(1)) has a positive answer over any
smooth scheme (over a field of characteristic $\neq 2$) is incorrect.
In these texts, the distinction between the groups
$\Iur^n(X)/\Iur^{n+1}(X)$ and $\unram{I}^n/\unram{I}^{n+1}(X)$ is not
made clear.
\end{remark}

% In this section we shall explain this seeming paradox.  While the
% groups $I_{\ur}^n/I_{\ur}^{n+1}(X)$ are canonically defined, the
% existence of $e_{\ur,\tot}$ guarenteed by Gille
% \cite[\S10.7]{gille:graded_Gersten-Witt} depends implicitly on the
% choice of an invertible $\OO_X$-module, and the results of Parimala,
% Scharlau, Sridharan show that this choice is indeed important.

\begin{defin}[{Parimala--Sridharan \cite{parimala_sridharan:graded_Witt}}]
A scheme $X$ over a field $F$ has the \linedef{extension property} for
quadratic forms if there exists $x_0 \in X(F)$ such that every regular
quadratic form on $X \setminus \{x_0\}$ extends to $X$.
\end{defin}

% \begin{prop}[{Parimala--Sridharan \cite[Lemma 4.2, 4.3]{parimala_sridharan:graded_Witt}}]
% Let $F$ be a field of characteristic $\neq 2$ and with $\cd_2 F \leq
% n$ and $X$ a smooth integral $k$-curve.  Then:
% \begin{enumerate}
% \item  If $X$ is affine then the
% unramified Milnor conjecture holds for $e_{\ur}^{m}$ with $m \geq n$.
% % $e^n :
% % I_{\ur}^n(X) \to \Hur^n(X,\muu_2^{\tensor n})$ is surjective.
% \end{enumerate}
% \end{prop}

\begin{prop}[{Parimala--Sridharan \cite[Lemma 4.3]{parimala_sridharan:graded_Witt}}]
Let $F$ be a field of characteristic $\neq 2$ and with $\cd_2 F \leq
2$ and $X$ a smooth integral $F$-curve.  Then the unramified Milnor
question for quadratic forms (Question \ref{question}(1)) has a
positive answer for $X$ if and only if $X$ has the extension property.
\end{prop}

The extension property is guaranteed when a residue theorem holds for
the Witt group.  The reside theorem for $X=\P^1$ is due to Milnor
\cite[\S5]{milnor:conjecture}.  For nonrational curves, the choice of
local uniformizers inherent in defining the residue maps is eliminated
by considering quadratic forms with values in the canonical bundle
$\omega_{X/F}$.

\begin{defin}
Let $X$ be a scheme and $\LL$ an invertible $\OO_X$-module.  An
\linedef{($\LL$-valued) symmetric bilinear form} on $X$ is a triple
$(\EE,b,\LL)$, where $\EE$ is a locally free $\OO_X$-module of finite
rank and $b : S^2\EE \to \LL$ is an $\OO_X$-module morphism.
\end{defin}

\begin{theorem}[{Geyer--Harder--Knebusch--Scharlau \cite{geyer_harder_knebusch_scharlau}}]
Let $X$ be a smooth proper integral curve over a field $F$ of
characteristic $\neq 2$. Then there is a canonical complex (which is
exact at the first two terms)
$$
\def\objectstyle{\scriptstyle} \def\labelstyle{\scriptstyle}
\xymatrix@C=18pt{
0 \ar[r] & 
W(X,\omega_{X/F}) \ar[r] &
W(F(X),\omega_{F(X)/F}) \ar[r]^(.45){\res^{\omega_X}} &
\smash{\underset{{x \in X^{(1)}}}{\textstyle\bigoplus}}
W(F(x),\omega_{F(x)/F}) \ar[r]^(.72){\Tr_{X/F}} &
 W(F)
}
$$
% $$
% 0 \to W(X,\omega_{X/F}) \to W(F(X),\omega_{F(X)/F})
% \mapto{\res^{\omega}} \bigoplus_{x \in X^{(1)}}
% W(F(x),\omega_{F(x)/F}) \mapto{\Tr_{X/F}} W(F)
% $$
and thus in particular $W(X,\omega_{X/F})$ has a residue theorem.
\end{theorem}

Now any choice of isomorphism $\vp: \NN^{\tensor 2} \isom
\omega_{X/F}$, induces a group isomorphism $W(X) \to
W(X,\omega_{X/F})$ via $(\EE,q) \mapsto (\EE\tensor\NN,\vp \circ
(q\tensor\id_{\NN}),\omega_{X/F})$.  Thus in particular, if
$\omega_{X/F}$ is a square in $\Pic(X)$, then $X$ has the extension
property.  Conversely:

\begin{theorem}[{Parimala--Sridharan \cite[Thm.\
3]{parimala_sridharan:nonsurjectivity}}] 
\label{thm:nonsurjectivity}
Let $F$ be a local field of characteristic $\neq 2$ and $X$ a smooth
integral hyperelliptic $F$-curve of genus $\geq 2$ with $X(F) \neq
\emptyset$.  Then the unramified Milnor question for quadratic forms
(Question \ref{question}(1)) holds over $X$
%$e_{\ur}^{2} : I_{\ur}^2(X) \to \Hur^2(X,\muu_2^{\tensor 2})$ is surjective 
if and only if $\omega_{X/k}$ is a square.
\end{theorem}

\begin{example}[{Parimala--Sridharan
\cite[Rem.\ 3]{parimala_sridharan:nonsurjectivity}}]
Let $X$ be the smooth proper hyperelliptic curve over $\Q_3$ with
affine model $y^2 = (x^2-3)(x^4+x^3+x^2+x+1)$.  One can show using
\cite[Thm.\ 2.4]{parimala_scharlau:extension} that $\omega_{X/F}$ is
not a square.  The point $(y,x)=(\sqrt{31},2)$ defines a
$\Q_3$-rational point of $X$.  Hence by Theorem
\ref{thm:nonsurjectivity}, the unramified Milnor question has a
negative answer over $X$.
\end{example}

Note that possible counter examples which are surfaces could be
extracted from the following result.   

\begin{theorem}[{Monnier \cite[Thm.\ 4.5]{monnier:unramified}}]
Let $X$ be a smooth integral surface over a real closed field $F$.
Then the unramified Milnor question for quadratic forms (Question
\ref{question}(1)) has a positive answer over $X$ if and only if the
cokernel of the mod 2 signature homomorphism is 4-torsion.
\end{theorem}

% \subsection{The canonical Gersten--Witt complex}

% If $X$ is a regular scheme, then as pointed out in Balmer--Walter
% \cite{balmer_walter:Gersten-Witt}, the Gersten--Witt complex
% canonically arises as $C(X,W,\omega_X)$
% $$
% 0 \to \bigoplus_{\eta \in X^{(0)}} W(\eta,\omega_{\eta/X}) 
%   \mapto{\res} \bigoplus_{x \in X^{(1)}} W(x,\omega_{x/X})
%   \mapto{\res} \bigoplus_{y \in X^{(2)}} W(y,\omega_{y/X}) \to \dotsm 
% $$
% where $W(x,\omega_{x/X})$ is the Witt group of finite length
% $\OO_{X,x}$-modules with duality $\Ext^p(-,\OO_{X,x})$ if $x \in
% X^{(p)}$ and the differentials arise as differentials in the
% Gersten--Wit complex.  It has a fundamental filtration
% $C(X,I^{\tot},\omega_X)$
% $$
% 0 \to \bigoplus_{\eta \in X^{(0)}} I^n(\eta,\omega_{\eta/X})
% \mapto{\res} \bigoplus_{x \in X^{(1)}} I^{n-1}(x,\omega_{x/X})
% \mapto{\res} \bigoplus_{y \in X^{(2)}} I^{n-2}(y,\omega_{y/X}) 
% \to \dotsm
% $$
% where $I^n(x,\omega_{x/X})$ is defined to be the image of
% $I^n(\resk(x))$ under the any choice of isomorphism $W(\resk(x)) \to
% W(x,\omega_{x/X})$, see Fasel \cite[\S9.2]{fasel:Chow-Witt}.
% Similarly, for any invertible $\OO_X$-module $\LL$, there's a complex
% $C(X,I^{\tot},\LL)$.

% \begin{lemma}[{Fasel \cite[Rem.\ 9.2.9]{fasel:Chow-Witt}}]
% For any invertible $\OO_X$-module $\LL$, the complexes
% $C(X,I^{\tot}/I^{\tot+1})$ and $C(X,I^{\tot}/I^{\tot+1},\LL)$ are
% canonically isomorphic. 
% \end{lemma}

% However, note that $C(X,I^{\tot})$ and $C(X,I^{\tot},\LL)$ may
% \emph{not} be isomorphic complexes!

\section{Line bundle-valued quadratic forms}
\label{sec:Counterexample_repaired}

Let $X$ be a smooth $F$-scheme.  Let $W(X,\LL)$ be the Witt group of
$\LL$-valued symmetric bilinear forms on $X$ and $W_{\total}(X) =
\bigoplus_{\LL} W(X,\LL)$ the \linedef{total} Witt group, where the
sum is taken over a set of representatives of $\Pic(X)/2$.  While this
group is only defined up to non-canonical isomorphism depending on our
choice of representatives, none of our cohomological invariants depend
on such isomorphisms, see \cite[\S1.2]{auel:surjectivity}.
Furthermore, we will not consider any ring structure on this group.
Thus we will not need to descend into most of the important
considerations of Balmer--Calm\`es \cite{balmer_calmes:lax}.

Let $e^0$ be the usual rank modulo 2 map
$$
e^0_{\total} : W_{\total}(X) \to \Z/2\Z = \Hur^0(X,\Z/2\Z)
$$
and $I^1_{\total}(X) = \oplus_{\LL} I^1(X,\LL)$ its kernel.  Then the
signed discriminant (see
\cite{parimala_sridharan:norms_and_pfaffians}) defines a surjective
map
$$
e^1_{\total} : I^1_{\total}(X) \to \Het^1(X,\muu_2) = \Hur^1(X,\muu_2).
$$
Finally, denote by $I^2(X,\LL) \subset I^1(X,\LL)$ the subgroup
generated by forms of trivial Arf invariant and $I_{\total}^2(X) =
\oplus_{\LL}I^2(X,\LL)$.  Then there exists a \linedef{total Clifford
invariant} for line bundle-valued quadratic forms
$$
e^2_{\total} : I^2_{\total}(X) \to \Brtwo(X) = \Hur^2(X,\muu_2)
$$ 
defined in \cite[Prop.\ 1.4]{auel:surjectivity}.  The surjectivity of
the total Clifford invariant can be viewed as a version of the global
Merkurjev question (Question \ref{question:GW}) for line bundle-valued
quadratic forms.  

\begin{theorem}[\cite{auel:surjectivity}]
Let $X$ be a smooth proper integral curve over a local field $F$ of
characteristic $\neq 2$ or a smooth proper integral surface
over a finite field $F$ of characteristic $\neq 2$.  Then the total
Clifford invariant
$$
e^{2}_{\total} : I_{\total}^2(X) \to \Brtwo(X)
$$
is surjective.
\end{theorem}
% \begin{proof}[Sketch of proof]
% By the period--index results of Saltman
% \cite{saltman:division_algebra_p-adic_curves} for curves over local
% fields and Lieblich \cite{lieblich:transcendence_2} for surfaces over
% (quasi-)finite fields, together with purity results for division
% algebras on surfaces, we show that $\Brtwo(X)$ is generated by Azumaya
% algebras of degree dividing 4.  Using the reduced norm form and
% reduced pfaffian form functors arising from the accidental isomorphism
% of algebraic groups $\mathsf{A}_1^2=\mathsf{D}_2$ and
% $\mathsf{A}_3=\mathsc{D}_3$, we can construct line bundle-valued
% quadratic forms with prescribed total Clifford invariant.
% \end{proof}

The surjectivity of the total Clifford invariant can also be
reinterpreted as the statement that while not every class in
$\unram{I}^2/\unram{I}^3(X) = \Hur^2(X)$ is represented by a quadratic
form on $X$, every class is represented by a line bundle-valued
quadratic form on $X$.

% Thus in the case of $p$-adic curves, the unramified Milnor conjecture
% holds for the ``total globalization'' of the Witt group.

% \nocite{}
% \bibliographystyle{amsplain}
% \bibliography{masterbib}

\begin{thebibliography}{100}

\bibitem{arason:quadratic_forms_invariant}
J.~Arason, \emph{Quadratic forms and {G}alois-cohomology}, S\'eminaire de
  {T}h\'eorie des {N}ombres, 1972--1973 ({U}niv. {B}ordeaux {I}, {T}alence),
  {E}xp. {N}o. 22, Lab. Th\'eorie des Nombres, Centre Nat. Recherche Sci.,
  Talence, 1973, p.~5.

\bibitem{arason:der_wittring}
J.~Arason, \emph{Der {W}ittring projektiver {R\"a}ume}, Math.\ {A}nn.
  \textbf{253} (1980), 205--212.

\bibitem{arason:cohomological_invariants}
J{\'o}n~Kr. Arason, \emph{Cohomologische invarianten quadratischer {F}ormen},
  J. Algebra \textbf{36} (1975), no.~3, 448--491.

\bibitem{arason:proof_Merkurjev_theorem}
\bysame, \emph{A proof of {M}erkurjev's theorem}, Quadratic and {H}ermitian
  forms ({H}amilton, {O}nt., 1983), CMS Conf. Proc., vol.~4, Amer. Math. Soc.,
  Providence, RI, 1984, pp.~121--130.

\bibitem{arason_elman_jacob:graded_Witt_I}
J{\'o}n~Kr. Arason, Richard Elman, and Bill Jacob, \emph{The graded {W}itt ring
  and {G}alois cohomology. {I}}, Quadratic and {H}ermitian forms ({H}amilton,
  {O}nt., 1983), CMS Conf. Proc., vol.~4, Amer. Math. Soc., Providence, RI,
  1984, pp.~17--50.

\bibitem{arason_elman_jacob:Witt_ring_elliptic_curve}
\bysame, \emph{On the {W}itt ring of elliptic curves}, {$K$}-theory and
  algebraic geometry: connections with quadratic forms and division algebras
  ({S}anta {B}arbara, {CA}, 1992), Proc. Sympos. Pure Math., vol.~58, Amer.
  Math. Soc., Providence, RI, 1995, pp.~1--25.

\bibitem{arason_elman_jacob:Witt_ring_elliptic_curve_local_field}
\bysame, \emph{The {W}itt ring of an elliptic curve over a local field}, Math.
  Ann. \textbf{306} (1996), no.~2, 391--418.

\bibitem{auel:surjectivity}
Asher Auel, \emph{Surjectivity of the total {C}lifford invariant and {B}rauer
  dimension}, arXiv:1108.5728v1, 2011.

\bibitem{baeza:semilocal_rings}
Ricardo Baeza, \emph{Quadratic forms over semilocal rings}, Lecture Notes in
  Mathematics, Vol. 655, Springer-Verlag, Berlin, 1978.

\bibitem{balmer:derived_witt_groups}
Paul Balmer, \emph{Derived {W}itt groups of a scheme}, J.\ Pure Appl.\ Algebra
  \textbf{141} (1999), 101--129.

\bibitem{balmer:triangular_witt_groups_I}
\bysame, \emph{Triangular {W}itt groups. {I}. {T}he 12-term exact sequence},
  {K}-Theory \textbf{19} (2000), 311--363.

\bibitem{balmer:triangular_witt_groups_II}
\bysame, \emph{Triangular {W}itt groups. {II}. {F}rom usual to derived}, Math.\
  Z. \textbf{236} (2001), 351--382.

\bibitem{balmer:handbook}
Paul Balmer, \emph{Witt groups}, Handbook of {$K$}-theory. {V}ol. 1, 2,
  Springer, Berlin, 2005, pp.~539--576.

\bibitem{balmer_calmes:lax}
Paul Balmer and Baptiste Calm{\`{e}}s, \emph{Bases of total {W}itt groups and
  lax-similitude}, preprint arXiv:1104.5051v1, April 2011.

\bibitem{balmer_walter:Gersten-Witt}
Paul Balmer and Charles Walter, \emph{A {G}ersten-{W}itt spectral sequence for
  regular schemes}, Ann. Sci. \'Ecole Norm. Sup. (4) \textbf{35} (2002), no.~1,
  127--152.

\bibitem{bass_tate:Milnor_global_field}
H.~Bass and J.~Tate, \emph{The {M}ilnor ring of a global field}, Algebraic
  {$K$}-theory, {II}: ``{C}lassical'' algebraic {$K$}-theory and connections
  with arithmetic ({P}roc. {C}onf., {S}eattle, {W}ash., {B}attelle {M}emorial
  {I}nst., 1972), Springer, Berlin, 1973, pp.~349--446. Lecture Notes in Math.,
  Vol. 342.

\bibitem{bass:algebraic_K-theory}
Hyman Bass, \emph{Algebraic {$K$}-theory}, W. A. Benjamin, Inc., New
  York-Amsterdam, 1968.

\bibitem{beilinson:pairings}
A.~A. Be{\u\i}linson, \emph{Height pairing between algebraic cycles},
  {$K$}-theory, arithmetic and geometry ({M}oscow, 1984--1986), Lecture Notes
  in Math., vol. 1289, Springer, Berlin, 1987, pp.~1--25.

\bibitem{bloch_ogus}
Spencer Bloch and Arthur Ogus, \emph{Gersten's conjecture and the homology of
  schemes}, Ann. Sci. \'Ecole Norm. Sup. (4) \textbf{7} (1974), 181--201
  (1975). \MR{0412191 (54 \#318)}

\bibitem{colliot:santa_barbara}
J.-L. Colliot-Th{\'e}l{\`e}ne, \emph{Birational invariants, purity and the
  {G}ersten conjecture}, {$K$}-theory and algebraic geometry: connections with
  quadratic forms and division algebras ({S}anta {B}arbara, {CA}, 1992), Proc.
  Sympos. Pure Math., vol.~58, Amer. Math. Soc., Providence, RI, 1995,
  pp.~1--64.

\bibitem{colliot-thelene_parimala:real_components}
J.-L. Colliot-Th{\'e}l{\`e}ne and R.~Parimala, \emph{Real components of
  algebraic varieties and \'etale cohomology}, Invent. Math. \textbf{101}
  (1990), no.~1, 81--99.

\bibitem{colliot-thelene_sansuc:fibres_quadratiques}
J.-L. Colliot-Th{\'e}l{\`e}ne and J.-J. Sansuc, \emph{Fibr\'es quadratiques et
  composantes connexes r\'eelles}, Math. Ann. \textbf{244} (1979), no.~2,
  105--134.

\bibitem{colliot-thelene_hoobler_kahn:Bloch-Ogus-Gabber}
Jean-Louis Colliot-Th{\'e}l{\`e}ne, Raymond~T. Hoobler, and Bruno Kahn,
  \emph{The {B}loch-{O}gus-{G}abber theorem}, Algebraic {$K$}-theory
  ({T}oronto, {ON}, 1996), Fields Inst. Commun., vol.~16, Amer. Math. Soc.,
  Providence, RI, 1997, pp.~31--94.

\bibitem{colliot-thelene_sansuc:type_multiplicatif}
Jean-Louis Colliot-Th{\'e}l{\`e}ne and Jean-Jacques Sansuc, \emph{Cohomologie
  des groupes de type multiplicatif sur les sch\'emas r\'eguliers}, C. R. Acad.
  Sci. Paris S\'er. A-B \textbf{287} (1978), no.~6, A449--A452.

\bibitem{dejong:surfaces}
A.~J. de~Jong, \emph{The period-index problem for the {B}rauer group of an
  algebraic surface}, Duke Math. J. \textbf{123} (2004), no.~1, 71--94.

\bibitem{dejong:gabber}
A.J. de~Jong, \emph{A result of {G}abber}, preprint, 2003.

\bibitem{dietel:Witt_ring_real_curves}
Gerhard Dietel, \emph{Wittringe singul\"arer reeller {K}urven. {I}, {II}},
  Comm. Algebra \textbf{11} (1983), no.~21, 2393--2448, 2449--2494.

\bibitem{elbaz-vincent_muller-stach}
Philippe Elbaz-Vincent and Stefan M{\"u}ller-Stach, \emph{Milnor {$K$}-theory
  of rings, higher {C}how groups and applications}, Invent. Math. \textbf{148}
  (2002), no.~1, 177--206.

\bibitem{E-K-V}
H{\'el\`e}ne Esnault, Bruno Kahn, and {E}ckart {V}iehweg, \emph{Coverings with
  odd ramification and {S}tiefel-{W}hitney classes}, J.\ reine Angew.\ Math.
  \textbf{441} (1993), 145--188.

\bibitem{esnault_kahn_levine_viehweg:arason}
H{\'e}l{\`e}ne Esnault, Bruno Kahn, Marc Levine, and Eckart Viehweg, \emph{The
  {A}rason invariant and mod {$2$} algebraic cycles}, J. Amer. Math. Soc.
  \textbf{11} (1998), no.~1, 73--118.

\bibitem{fasel:Chow-Witt}
Jean Fasel, \emph{Groupes de {C}how-{W}itt}, M\'em. Soc. Math. Fr. (N.S.)
  (2008), no.~113, viii+197.

\bibitem{fernandez-carmena:Witt_group_surfaces}
Fernando Fern{\'a}ndez-Carmena, \emph{On the injectivity of the map of the
  {W}itt group of a scheme into the {W}itt group of its function field}, Math.
  Ann. \textbf{277} (1987), no.~3, 453--468.

\bibitem{friedlander:motivic_complexes}
Eric~M. Friedlander, \emph{Motivic complexes of {S}uslin and {V}oevodsky},
  Ast\'erisque (1997), no.~245, Exp.\ No.\ 833, 5, 355--378.

\bibitem{friedlander_rapoport_suslin:fields_medal_2002}
Eric~M. Friedlander, Michael Rapoport, and Andrei Suslin, \emph{The
  mathematical work of the 2002 {F}ields medalists}, Notices Amer. Math. Soc.
  \textbf{50} (2003), no.~2, 212--217.

\bibitem{frohlich:unimodular}
A.~Fr{\"o}hlich, \emph{On the {$K$}-theory of unimodular forms over rings of
  algebraic integers}, Quart. J. Math. Oxford Ser. (2) \textbf{22} (1971),
  401--423.

\bibitem{gabber:brauer}
Ofer Gabber, \emph{Some theorems on {A}zumaya algebras}, The {B}rauer group
  ({S}em., {L}es {P}lans-sur-{B}ex, 1980), Lecture Notes in Math., vol. 844,
  Springer, Berlin, 1981, pp.~129--209.

\bibitem{geisser:K-theory_handbook}
Thomas Geisser, \emph{Motivic cohomology, {$K$}-theory and topological cyclic
  homology}, Handbook of {$K$}-theory. {V}ol. 1, 2, Springer, Berlin, 2005,
  pp.~193--234.

\bibitem{geisser_levine:bloch-kato}
Thomas Geisser and Marc Levine, \emph{The {B}loch-{K}ato conjecture and a
  theorem of {S}uslin-{V}oevodsky}, J. Reine Angew. Math. \textbf{530} (2001),
  55--103.

\bibitem{geyer_harder_knebusch_scharlau}
W.-D. Geyer, G.~Harder, M.~Knebusch, and W.~Scharlau, \emph{{E}in
  {R}esiduensatz f{\"{u}}r symmetrische {B}ilinearformen}, Invent. Math.
  \textbf{11} (1970), 319--328.

\bibitem{gille:support}
Stefan Gille, \emph{On {W}itt groups with support}, Math. Ann. \textbf{322}
  (2002), no.~1, 103--137.

\bibitem{gille:graded_Gersten-Witt}
\bysame, \emph{A graded {G}ersten-{W}itt complex for schemes with a dualizing
  complex and the {C}how group}, J. Pure Appl. Algebra \textbf{208} (2007),
  no.~2, 391--419.

\bibitem{grothendieck:classes_chern_representations}
Alexander Grothendieck, \emph{Dix expos{\'{e}}s sur la cohomologie des
  sch{\'{e}}mas}, ch.~{VIII} {C}lasses de {C}hern et repr{\'{e}}sentations
  lin{\'{e}}aires des groupes discrets, pp.~215--305, North-{H}olland
  {P}ublishing {C}ompany, {A}msterdam; {M}asson {\&} {C}ie, {\'{E}}diteur,
  Paris, 1968.

\bibitem{guin:homologie_GL_Milnor_K-theory}
Daniel Guin, \emph{Homologie du groupe lin\'eaire et {$K$}-th\'eorie de
  {M}ilnor des anneaux}, J. Algebra \textbf{123} (1989), no.~1, 27--59.

\bibitem{harder:Witt_group_curves}
G{\"u}nter Harder, \emph{Halbeinfache {G}ruppenschemata \"uber vollst\"andigen
  {K}urven}, Invent. Math. \textbf{6} (1968), 107--149.

\bibitem{hoobler:Merkurjev-Suslin_semilocal_ring}
Raymond~T. Hoobler, \emph{The {M}erkuriev-{S}uslin theorem for any semi-local
  ring}, J. Pure Appl. Algebra \textbf{207} (2006), no.~3, 537--552.

\bibitem{jacob_rost:invariant_degree_4}
Bill Jacob and Markus Rost, \emph{Degree four cohomological invariants for
  quadratic forms}, Invent. Math. \textbf{96} (1989), no.~3, 551--570.

\bibitem{kahn:milnor_conjecture_article}
Bruno Kahn, \emph{La conjecture de {M}ilnor (d'apr\`es {V}. {V}oevodsky)},
  Ast\'erisque (1997), no.~245, Exp.\ No.\ 834, 5, 379--418, S{\'e}minaire
  Bourbaki, Vol. 1996/97.

\bibitem{kahn_sujatha:motivic_cohomology_unramified_quadrics}
Bruno Kahn and R.~Sujatha, \emph{Motivic cohomology and unramified cohomology
  of quadrics}, J. Eur. Math. Soc. (JEMS) \textbf{2} (2000), no.~2, 145--177.

\bibitem{kerz:Gersten_conjecture_Milnor_K-theory}
Moritz Kerz, \emph{The {G}ersten conjecture for {M}ilnor {$K$}-theory}, Invent.
  Math. \textbf{175} (2009), no.~1, 1--33.

\bibitem{kerz:Milnor_K-theory_local_rings}
\bysame, \emph{Milnor {$K$}-theory of local rings with finite residue fields},
  J. Algebraic Geom. \textbf{19} (2010), no.~1, 173--191. \MR{2551760
  (2010j:19006)}

\bibitem{kerz:Milnor-Chow_homomorphism}
Moritz Kerz and Stefan M{\"u}ller-Stach, \emph{The {M}ilnor-{C}how homomorphism
  revisited}, $K$-Theory \textbf{38} (2007), no.~1, 49--58.

\bibitem{knebusch}
M.\ Knebusch, \emph{Symmetric bilinear forms over algebraic varieties},
  {C}onference on {Q}uadratic {F}orms--1976 ({K}ingston, {O}nt.) (G.\ Orzech,
  ed.), {Q}ueen's {P}apers in {P}ure and {A}ppl.\ {M}ath., no.~46, {Q}ueen's
  Univ., 1977, pp.~103--283.

\bibitem{knebusch_scharlau:reciprocity}
M.~Knebusch and W.~Scharlau, \emph{Quadratische {F}ormen und quadratische
  {R}eziprozit\"atsgesetze \"uber algebraischen {Z}ahlk\"orpern}, Math. Z.
  \textbf{121} (1971), 346--368.

\bibitem{knebusch:heidelberg}
Manfred Knebusch, \emph{Grothendieck- und {W}ittringe von nichtausgearteten
  symmetrischen {B}ilinearformen}, S.-B. Heidelberger Akad. Wiss. Math.-Natur.
  Kl. \textbf{1969/70} (1969/1970), 93--157.

\bibitem{knebusch:curves_real_fields_II}
\bysame, \emph{On algebraic curves over real closed fields. {II}}, Math. Z.
  \textbf{151} (1976), no.~2, 189--205.

\bibitem{knus_parimala_sridharan:compositions_triality}
M.-A. Knus, R.~Parimala, and R.~Sridharan, \emph{On compositions and triality},
  J. Reine Angew. Math. \textbf{457} (1994), 45--70.

\bibitem{knus:quadratic_hermitian_forms}
Max-Albert Knus, \emph{Quadratic and hermitian forms over rings},
  {S}pringer-{V}erlag, {B}erlin, 1991.

\bibitem{knus_ojanguren:metabolic}
Max-Albert Knus and Manuel Ojanguren, \emph{The {C}lifford algebra of a
  metabolic space}, Arch. Math. (Basel) \textbf{56} (1991), no.~5, 440--445.

\bibitem{lam:algebraic_theory_quadratic_forms}
T.~Y. Lam, \emph{The algebraic theory of quadratic forms}, Benjamin/Cummings
  Publishing Co. Inc. Advanced Book Program, Reading, Mass., 1980, Revised
  second printing, Mathematics Lecture Note Series.

\bibitem{lichtenbaum:values}
S.~Lichtenbaum, \emph{Values of zeta-functions at nonnegative integers}, Number
  theory, {N}oordwijkerhout 1983 ({N}oordwijkerhout, 1983), Lecture Notes in
  Math., vol. 1068, Springer, Berlin, 1984, pp.~127--138.

\bibitem{lieblich:moduli_twisted_sheaves}
Max Lieblich, \emph{Moduli of twisted sheaves}, Duke Math. J. \textbf{138}
  (2007), no.~1, 23--118.

\bibitem{merkurjev:degree_2}
A.~S. Merkur{'}ev, \emph{On the norm residue symbol of degree {$2$}}, Dokl.
  Akad. Nauk SSSR \textbf{261} (1981), no.~3, 542--547.

\bibitem{merkurjev:another_proof_norm_2}
Alexander Merkurjev, \emph{On the norm residue homomorphism of degree two},
  Proceedings of the {S}t. {P}etersburg {M}athematical {S}ociety. {V}ol. {XII}
  (Providence, RI), Amer. Math. Soc. Transl. Ser. 2, vol. 219, Amer. Math.
  Soc., 2006, pp.~103--124.

\bibitem{merkurjev_suslin:norm3}
A.S. Merkurjev and A.A. Suslin, \emph{The norm residue homomorphism of degree
  3}, Math. USSR Izv. \textbf{36} (1991), 349--368.

\bibitem{milne:etale_cohomology}
{J.\ S}.\ {M}ilne, \emph{{\'E}tale cohomology}, Princeton {M}athematical
  {S}eries, no.~33, Princeton {U}niversity {P}ress, Princeton, N.J., 1980.

\bibitem{milnor:conjecture}
John Milnor, \emph{Algebraic {$K$}-theory and quadratic forms}, Invent. Math.
  \textbf{9} (1969/1970), 318--344.

\bibitem{monnier:unramified}
Jean-Philippe Monnier, \emph{Unramified cohomology and quadratic forms}, Math.
  Z. \textbf{235} (2000), no.~3, 455--478.

\bibitem{morel:voevodsky}
F.~Morel, \emph{Voevodsky's proof of {M}ilnor's conjecture}, Bull. Amer. Math.
  Soc. (N.S.) \textbf{35} (1998), no.~2, 123--143.

\bibitem{morel:proof_milnor}
Fabien Morel, \emph{Milnor's conjecture on quadratic forms and mod 2 motivic
  complexes}, Rend. Sem. Mat. Univ. Padova \textbf{114} (2005), 63--101 (2006).

\bibitem{nesterenko_suslin}
Yu.~P. Nesterenko and A.~A. Suslin, \emph{Homology of the general linear group
  over a local ring, and {M}ilnor's {$K$}-theory}, Izv. Akad. Nauk SSSR Ser.
  Mat. \textbf{53} (1989), no.~1, 121--146.

\bibitem{ojanguren_parimala_sridharan_suresh:purity_threefolds}
M.~Ojanguren, R.~Parimala, R.~Sridharan, and V.~Suresh, \emph{Witt groups of
  the punctured spectrum of a 3-dimensional regular local ring and a purity
  theorem}, J. London Math. Soc. (2) \textbf{59} (1999), no.~2, 521--540.

\bibitem{ojanguren_panin:purity}
Manuel Ojanguren and Ivan Panin, \emph{A purity theorem for the {W}itt group},
  Ann. Sci. \'Ecole Norm. Sup. (4) \textbf{32} (1999), no.~1, 71--86.

\bibitem{orlov_vishik_voevodsky}
D.~Orlov, A.~Vishik, and V.~Voevodsky, \emph{An exact sequence for
  {$K^M_\ast/2$} with applications to quadratic forms}, Ann. of Math. (2)
  \textbf{165} (2007), no.~1, 1--13.

\bibitem{pardon:filtered}
William Pardon, \emph{The filtered {G}ersten--{W}itt resolution for regular
  schemes}, preprint, K-theory Preprint Archives,
  http://www.math.uiuc.edu/K-theory/0419/, May 2000.

\bibitem{parimala:affine_three_folds}
R.~Parimala, \emph{Witt groups of affine three-folds}, Duke Math. J.
  \textbf{57} (1988), no.~3, 947--954.

\bibitem{parimala:Witt_groups_conics}
\bysame, \emph{Witt groups of conics, elliptic, and hyperelliptic curves}, J.
  Number Theory \textbf{28} (1988), no.~1, 69--93. \MR{925609 (89a:14028)}

\bibitem{parimala_scharlau:extension}
R.~Parimala and W.~Scharlau, \emph{On the canonical class of a curve and the
  extension property for quadratic forms}, Recent advances in real algebraic
  geometry and quadratic forms ({B}erkeley, {CA}, 1990/1991; {S}an {F}rancisco,
  {CA}, 1991), Contemp. Math., vol. 155, Amer. Math. Soc., Providence, RI,
  1994, pp.~339--350.

\bibitem{parimala_sridharan:graded_Witt}
R.~Parimala and R.~Sridharan, \emph{Graded {W}itt ring and unramified
  cohomology}, $K$-Theory \textbf{6} (1992), no.~1, 29--44.

\bibitem{parimala_sridharan:nonsurjectivity}
\bysame, \emph{Nonsurjectivity of the {C}lifford invariant map}, Proc. Indian
  Acad. Sci. Math. Sci. \textbf{104} (1994), no.~1, 49--56, K. G. Ramanathan
  memorial issue.

\bibitem{parimala_sridharan:norms_and_pfaffians}
R.\ Parimala and R.\ Sridharan, \emph{Reduced norms and pfaffians via
  {B}rauer-{S}everi schemes}, Contemp.\ Math. \textbf{155} (1994), 351--363.

\bibitem{parimala_srinivas:brauer_group_involution}
R.\ Parimala and V.\ Srinivas, \emph{Analogues of the {B}rauer group for
  algebras with involution}, Duke Math. J. \textbf{66} (1992), no.~2, 207--237.

\bibitem{parimala_sujatha:witt_group_hyperelliptic_curve}
R.~Parimala and R.~Sujatha, \emph{Witt group of hyperelliptic curves}, Comment.
  Math. Helv. \textbf{65} (1990), no.~4, 559--580.

\bibitem{parimala:curves_local_fields}
Raman Parimala, \emph{Witt groups of curves over local fields}, Comm. Algebra
  \textbf{17} (1989), no.~11, 2857--2863.

\bibitem{pfister:survey}
A.~Pfister, \emph{On the {M}ilnor conjectures: history, influence,
  applications}, Jahresber. Deutsch. Math.-Verein. \textbf{102} (2000), no.~1,
  15--41.

\bibitem{pfister:habilitationsschrift}
Albrecht Pfister, \emph{Quadratische {F}ormen in beliebigen {K}\"orpern},
  Invent. Math. \textbf{1} (1966), 116--132. \MR{0200270 (34 \#169)}

\bibitem{pfister:historical_remarks}
\bysame, \emph{Some remarks on the historical development of the algebraic
  theory of quadratic forms}, Quadratic and {H}ermitian forms ({H}amilton,
  {O}nt., 1983), CMS Conf. Proc., vol.~4, Amer. Math. Soc., Providence, RI,
  1984, pp.~1--16.

\bibitem{quebbemann_scharlau_et_al}
H.~G. Quebbemann, R.~Scharlau, W.~Scharlau, and M.~Schulte, \emph{Quadratische
  {F}ormen in additiven {K}ategorien}, Bull. Soc. Math. France Suppl. Mem.
  (1976), no.~48, 93--101, Colloque sur les Formes Quadratiques (Montpellier,
  1975).

\bibitem{quebbemann_scharlau_schulte}
H.-G. Quebbemann, W.~Scharlau, and M.~Schulte, \emph{Quadratic and {H}ermitian
  forms in additive and abelian categories}, J. Algebra \textbf{59} (1979),
  no.~2, 264--289.

\bibitem{rost:K_3}
Markus Rost, \emph{On {H}ilbert {S}atz 90 for $k_3$ for degree-two extensions},
  preprint, May 1986.

\bibitem{rost:cycle_modules}
Markus Rost, \emph{Chow groups with coefficients}, Doc. Math. \textbf{1}
  (1996), No. 16, 319--393 (electronic).

\bibitem{scharlau:book}
Winfried Scharlau, \emph{Quadratic and {H}ermitian forms}, Grundlehren der
  Mathematischen Wissenschaften [Fundamental Principles of Mathematical
  Sciences], vol. 270, Springer-Verlag, Berlin, 1985.

\bibitem{serre:cohomologie_galoisienne}
Jean-Pierre Serre, \emph{{C}ohomologie galoisienne}, 5th ed., 2nd print.,
  {L}ecture {N}otes in {M}ath., vol.~5, Springer-Verlag, Berlin, 1997.

\bibitem{shyevski:fifth_invariant}
M.~Shyevski, \emph{The fifth invariant of quadratic forms}, Dokl. Akad. Nauk
  SSSR \textbf{308} (1989), no.~3, 542--545.

\bibitem{sujatha:Witt_groups_real_surfaces}
R.~Sujatha, \emph{Witt groups of real projective surfaces}, Math. Ann.
  \textbf{288} (1990), no.~1, 89--101.

\bibitem{suslin_voevodsky}
Andrei Suslin and Vladimir Voevodsky, \emph{Bloch-{K}ato conjecture and motivic
  cohomology with finite coefficients}, The arithmetic and geometry of
  algebraic cycles ({B}anff, {AB}, 1998), NATO Sci. Ser. C Math. Phys. Sci.,
  vol. 548, Kluwer Acad. Publ., Dordrecht, 2000, pp.~117--189.

\bibitem{thomason:nonexistence}
R.~W. Thomason, \emph{Le principe de scindage et l'inexistence d'une
  {$K$}-theorie de {M}ilnor globale}, Topology \textbf{31} (1992), no.~3,
  571--588.

\bibitem{totaro:Milnor_K-theory}
Burt Totaro, \emph{Milnor {$K$}-theory is the simplest part of algebraic
  {$K$}-theory}, $K$-Theory \textbf{6} (1992), no.~2, 177--189.

\bibitem{totaro:non-injectivity}
\bysame, \emph{Non-injectivity of the map from the {W}itt group of a variety to
  the {W}itt group of its function field}, J. Inst. Math. Jussieu \textbf{2}
  (2003), no.~3, 483--493.

\bibitem{voevodsky:Milnor_conjecture_I}
Vladimir Voevodsky, \emph{Motivic cohomology with {${\bf Z}/2$}-coefficients},
  Publ. Math. Inst. Hautes \'Etudes Sci. (2003), no.~98, 59--104.

\bibitem{wadsworth:proof_Merkurjev_theorem}
Adrian~R. Wadsworth, \emph{Merkurjev's elementary proof of {M}erkurjev's
  theorem}, Applications of algebraic {$K$}-theory to algebraic geometry and
  number theory, {P}art {I}, {II} ({B}oulder, {C}olo., 1983), Contemp. Math.,
  vol.~55, Amer. Math. Soc., Providence, RI, 1986, pp.~741--776.

\bibitem{walter:GW_groups_triangulated_categories}
Charles Walter, \emph{{G}rothendieck-{W}itt groups of triangulated categories},
  preprint, {K}-theory preprint archive, 2003.

\bibitem{witt:quadratic}
E.~Witt, \emph{{Theorie der quadratischen Formen in beliebigen K\"orpern.}}, J.
  reine angew. Math. \textbf{176} (1936), 31--44 (German).

\bibitem{zainoulline:purity_witt_group}
Kirill Zainoulline, \emph{Witt groups of varieties and the purity problem},
  Quadratic forms, linear algebraic groups, and cohomology, Dev. Math.,
  vol.~18, Springer, New York, 2010, pp.~173--185.

\end{thebibliography}

%\end{document}

\providecommand{\bysame}{\leavevmode\hbox to3em{\hrulefill}\thinspace}
% \providecommand{\MR}{\relax\ifhmode\unskip\space\fi MR }
% % \MRhref is called by the amsart/book/proc definition of \MR.
% \providecommand{\MRhref}[2]{%
%   \href{http://www.ams.org/mathscinet-getitem?mr=#1}{#2}
% }
\providecommand{\href}[2]{#2}

\end{document}